\documentclass[10pt,reqno]{amsart}
\usepackage{amsmath,amssymb,latexsym,soul,cite,mathrsfs}
\usepackage{color,enumitem,graphicx}
\usepackage[colorlinks=true,urlcolor=blue,
citecolor=red,linkcolor=blue,linktocpage,pdfpagelabels,
bookmarksnumbered,bookmarksopen]{hyperref}
\usepackage[english]{babel}
\usepackage{enumitem}

\usepackage[left=1.5cm,right=1.5cm,top=1.8cm,bottom=1.8cm]{geometry}

\usepackage[hyperpageref]{backref}

\makeatletter
\newcommand{\leqnomode}{\tagsleft@true}
\newcommand{\reqnomode}{\tagsleft@false}
\makeatother

\date{}
\def\nd{\noindent}
\def\thend{\rule{3mm}{3mm}}

\newtheorem{theorem}{Theorem}[section]
\newtheorem{definition}{Definition}[section]
\newtheorem{cor}{Corollary}[section]
\newtheorem{prop}{Proposition}[section]
\newtheorem{lemma}{Lemma}[section]
\newtheorem{rmk}{Remark}[section]
\newtheorem{exm}{Example}[section]

\newcommand{\R}{\mathbb{R}}

\begin{document}
	\title[Singular Stein-Weiss problems]{Multiplicity of solutions for singular elliptic problems with Stein-Weiss term}
	\vspace{1cm}

    \author{Márcia S. B. A. Cardoso}
	\address{Márcia S. B. A. Cardoso \newline Universidade Federal de Goias, IME, Goi\^ania-GO, Brazil }
	\email{\tt marcia.cardoso@ifg.edu.br}
	
	\author{Marcos. L. M. Carvalho}
	\address{M. L. M. Carvalho \newline Universidade Federal de Goias, IME, Goi\^ania-GO, Brazil }
	\email{\tt marcos$\_$leandro$\_$carvalho@ufg.br}

    \author{Edcarlos D. Silva}
	\address{Edcarlos D da Silva \newline  Universidade Federal de Goias, IME, Goi\^ania-GO, Brazil}
	\email{\tt edcarlos@ufg.br}

    \author{Minbo Yang}
	\address{Minbo Yang \newline School of Mathematical Sciences, Zhejiang Normal University, Jinhua 321004, Zhejiang,
People’s Republic of China }
	\email{\tt mbyang@zjnu.edu.cn}

	
	\thanks{The second author was partially supported by CNPq with grants 309026/2020-2. CNPq partially supported the fourth author with grant 316643/2021-1. Minbo Yang was partially supported by NSFC (12471114) and ZJNSF(LZ22A010001).}
    \thanks{Corresponding author: Edcarlos D. Silva, email: edcarlos@ufg.br}

	\begin{abstract}
        In the present work, we establish the existence and multiplicity of positive solutions for the singular elliptic equations with a double weighted nonlocal interaction term defined in the whole space $\mathbb{R}^N$. The nonlocal term and the fact that the energy functional is not differentiable are the main difficulties for this kind of problem. We apply the Nehari method and the nonlinear Rayleigh quotient to prove that our main problem has at least two positive weak solutions. Furthermore, we prove a nonexistence result related to the extreme $\lambda^*> 0$ given by the nonlinear Rayleigh quotient.
    \end{abstract}
	
	\maketitle
	
	\section{Introduction}

In this work, we consider the existence and multiplicity of positive weak solutions for singular elliptic problems taking into account a double weighted nonlocal term. More precisely, we shall consider the following nonlocal elliptic problem:
\begin{equation} \label{P1}
  \left\{\begin{array}{ll}
    - \Delta u +V(x) u =   \displaystyle\lambda \frac{a(x)}{u^{1-q}}  + \displaystyle \int \limits_{\mathbb{R}^N}\frac{b(y)\vert u(y) \vert^p dy}{\vert x\vert^\alpha\vert x-y\vert^\mu \vert y\vert^\alpha} b(x)\vert u\vert^{p-2}u,\ \  \hbox{in}\ \mathbb{R}^N,
    \\
    u>0,~u\in H^1(\mathbb{R}^N).
\end{array}
\right.  \tag{$Q$}
\end{equation}
where $0<q<1$, $\lambda > 0$, $p \in (2_{\alpha,\mu}, 2_{\alpha,\mu}^*)$ where $\mu>0,0< 2\alpha+\mu< N, \,\,\alpha\in (0,N), 2_{\alpha,\mu}=(2N-2\alpha-\mu)/N$ and $2_{\alpha,\mu}^*=(2N-2\alpha-\mu)/(N-2)$. Later on, we shall discuss some extra assumptions on the continuous potentials $a, V: \mathbb{R}^N \to \mathbb{R}$.
The nonlocal elliptic equation with such kind of double weighted nonlocal term was introduced by \cite{minbo, Yang}. This type of problem can be understood as the Euler-Lagrange equation of a minimizing problem related to the Stein-Weiss problem introduced in \cite{Stein_weiss1958}. Furthermore, assuming that $\alpha = 0$, we obtain the singular Choquard problem which can be written in the following form:
\begin{equation} \label{C}
  \left\{\begin{array}{ll}
    - \Delta u +V(x) u =   \displaystyle\lambda \frac{a(x)}{u^{1-q}}  + \displaystyle \int \limits_{\mathbb{R}^N}\frac{b(y)\vert u(y) \vert^p dy}{\vert x-y\vert^\mu  } b(x)\vert u\vert^{p-2}u,\ \  \hbox{in}\ \mathbb{R}^N,
    \\
    u>0,~u\in H^1(\mathbb{R}^N).
\end{array}
\right.  \tag{$C$}
\end{equation}
where $0<q<1$, $\lambda > 0, (N + \mu)/N < p < (N + \mu)/( N - 2), N \geq 3$ and $a, b: \mathbb{R}^N \to \mathbb{R}$ are continuous potentials satisfying some extra assumptions. Many results on the existence, multiplicity, and symmetry of solutions for the Choquard problem have been obtained in the last years depending on the assumptions on the nonlinearity and the potential $V$. On this subject, we refer the reader to the works \cite{Alves,Alves3,RAY2021,Ding,Gao1,Gao3}. The main obstacle in this kind of problem is to control the behavior of the convolution term taking into account the potential $b: \mathbb{R}^N \to \mathbb{R}$ and the fact that $\mu \in (0, N)$. We refer the reader also to the works \cite{Moroz_Van,MOROZ,Moroz2}. The Choquard equation has been widely considered in the last decades and it has also several physical applications. For instance, assuming that $N=3, \mu=1, V\equiv 1,p=2,\lambda=0$ and $b=1$, the Problem \eqref{C} boils down to the Choquard-Pekar equation
	\begin{equation*}
		- \Delta u + u =\left(\displaystyle\frac{1}{\vert x\vert}*\vert u\vert^2\right)u\ \  \hbox{in}\ \mathbb{R}^3, 	u\in H^1(\mathbb{R}^3).
	\end{equation*}
These kind of problems were introduced in \cite{Pekar} assuming that $N =3, \lambda = 0, \alpha = 0$ and $\mu = 1$. The main objective of that work was to describe a polaron at rest in quantum field theory. The same problem is also used to describe an electron trapped in its own hole. The main idea here is to use an approximation of Hartree-Fock theory of one-component plasma. On this subject we refer the interested reader to \cite{Penrose} where it was also proposed a model of self-gravitating matter in which the reduction of the quantum state is considered a gravitational phenomenon. In this context, the problem is known as the nonlinear Schr\"odinger-Newton equation.

It is necessary to emphasize that the singular elliptic problems have been widely studied in the last years, see \cite{Coclite,Fulks,kaye,Yao,Yijing,Yijing2008}. In \cite{Yijing} the authors considered a class of singular elliptic problems with Dirichlet boundary conditions where the nonlinearity considers the effects of a superlinear term and a singular term. This problem allows us to establish the existence and multiplicity of solutions taking into account a parameter
$\lambda > 0$ in front of the singular term. In that work, the authors considered the following singular elliptic problem:
\begin{equation}\label{Yijing}\ \
 \left\{\begin{array}{ll}
    - \Delta u=\lambda u^{\beta} +p(x) u^{-\gamma},\;\;u\in \Omega,\\
    u>0,\;\;\hbox{in}\;\;\Omega,\\
    u=0,\;\;\hbox{on}\;\;\partial\Omega,
\end{array}\right. \tag{$Q1$}
\end{equation}
where $\Omega \subset \mathbb{R}^N$ is a bounded domain, $p:\Omega \to \mathbb{R}$ is a non-negative and non-trivial function in $L^2(\Omega)$, $1 < \beta < 2^* - 1$, $0 < \gamma < 1$, $2^* = 2N/(N-2)$ is the usual critical Sobolev exponent, $N \geq 3$, and $\lambda > 0$ is a parameter. In \cite{Yijing2008} the authors considered the following singular elliptic problem:
\begin{equation}\label{Yijing2008}\ \
 \left\{\begin{array}{ll}
    - \Delta u=\lambda a(x) u^{-\gamma} +b(x) u^{p},\;\;u\in \Omega,\\
    u>0,\;\;\hbox{in}\;\;\Omega,\\
    u=0,\;\;\hbox{on}\;\;\partial\Omega,
\end{array}\right. \tag{$Q2$}
\end{equation}
where $a, b \in C(\overline{\Omega})$, $a > 0$ and $b$ can be a sing-changing function. In that work, the authors proved a result showing that there exists $\lambda^* > 0$ such that Problem \eqref{Yijing2008} admits at least two non-negative solutions for each $\lambda \in (0, \lambda^*)$. On this subject, for bounded domains, we refer the interesting reader to \cite{NS, Rabinowitz, Lazer1991, Yijing} where different techniques and many types of nonlinearities are considered. On this topic taking into account unbounded domains, we refer the reader to \cite{kaye} and references therein.

Many researchers have been attracted due to the pioneering on singular elliptic problems introduced in \cite{Fulks}.  For example, in \cite{Yijing,Coclite} the authors proved that there exists $\Lambda^* \in (0, \infty)$ such that Problem \eqref{Yijing} has at least one solution for each $\lambda  \in (0,\Lambda^*)$ and does not exist any solution for $\lambda > \Lambda^* $ whenever $p \equiv 1$ in $\overline{\Omega}$. The authors showed that there is a unique solution under extra assumptions on $p$. It is worth noting that in \cite{Yao, Wiegner} the existence of a unique positive solution was proved in the cases $b \equiv 1$ or $0 < b < 1$.
In \cite{kaye}, the authors studied singular elliptic problems proving the existence and multiplicity of positive solutions for the following problem:
\begin{equation*}\label{Kaye}
    -\Delta u+V(x)u=\lambda a(x)u^{-\gamma}+b(x)u^p,\;\;\hbox{in}\;\;\mathbb R^N.
\end{equation*}
In that work, it was investigated via the Nehari method together with the nonlinear Rayleigh quotient which was introduced in \cite{YAVDAT2017}. Here we observe that in \cite{kaye} the potential $b \in L^{\infty}(\mathbb{R}^N)$ is a sign-changing function, $0 < a \in L^{\frac{2}{1+\gamma}}(\mathbb{R}^N)$, $V$ is a continuous positive function, $N \geq 3$, and $\lambda > 0$ is a parameter.

In order to deal with the singular elliptic problems by variational methods the main difficulty arises from the lack of regularity due to the fact that the energy functional is not in class $C^1$. Moreover, by using the fibering map and the Nehari method for our case, we are able to prove the main results by the fact that the functional is only of $C^0$ class. Hence, some classical tools can be not applied in order to ensure the existence and the multiplicity of critical points for the energy functional. Here, we shall apply the Nehari method together with the nonlinear Rayleigh quotient which is only in $C^0$ class due to the singular term. The nonlinear Raleigh quotient was also applied in many other situations, see for instance \cite{RAY2021,Tpar,MYC,YAVDAT2017}. It is important to emphasize that depending on the size of $\lambda > 0$ inflection points for the fibering map is allowed. This kind of problem is overcome using some fine estimates together with some topological structures of the Nehari set.

As we all know, a nonzero critical point $u$ is said to be a ground state solution if $u$ has the minimal energy among any others nontrivial solutions. Furthermore, $u$ is a bound state solutions whenever $u$ is a critical point for the energy functional which has finite energy.

Recall that a point $u$ is called as a degenerated point whenever the second derivative for the energy functional at $u$ is zero. Otherwise, $u$ is said to be a nondegenerate point. It is important to stress that using the Nehari set we are able to prove that there exist minimizers in suitable energy level.  Assuming that $u$ is a nondenegerate point the Lagrange Multiplier Theorem says that the $u$ is a critical point for the energy functional. However, assuming that $u$ is a degenerated point, the Lagrange Multiplier Theorem does not anymore showing that the Nehari set is not a natural constraint. Hence, the main idea is to avoid degenerate points which are also minimizer in the Nehari set. It is worthwhile to mention that for $\lambda = \lambda^*$ there exists degenerate point for the energy functional. In fact, $\lambda^*$ is the smallest positive number with this property showing that the associated fibering map admits also inflections points. For further results on this subject we refer the reader to \cite{RAY2021,Tpar,MYC}.

 Our main contribution in the present work is to study the existence and multiplicity of solutions for singular elliptic problems taking into account a Stein-Weiss term.
 The main feature here is to prove the existence of a parameter $\lambda^*$ such that our main problem has at least two solutions for each $\lambda \in (0,\lambda^*)$. A solution is a ground state solution while the another one is bound state solution.  In fact, $\lambda^*$ is the largest positive number in such a way that the Nehari method can be applied directly. The main purpose here is to prove that inflection points are not verified for the fibering map in the associated energy functional. Furthermore, we prove a nonexistence result for $\lambda = \lambda^*$. The last assertion enables us to ensure that Problem \eqref{P1} admits at least two positive solutions for $\lambda = \lambda^*$. In order to do that we consider a sequence $(\lambda_k)$ such that $\lambda_k \to \lambda^*$ and $\lambda_k < \lambda^*$. The main strategy is to consider two minimizer sequences $(u_k)$ and $(w_k)$ in the Nehari sets showing that these sequences strongly converge. Moreover, we consider a second extreme $\lambda_* < \lambda^*$ which permits us to classify the sign of the energy functional for any nonzero weak solution for the Problem \eqref{P1}.

Below, we summarize the main contributions of this work to the current literature.
\begin{enumerate}
\item As far as we know, this is the first work driven by a Stein-Weiss term defined in the whole $\mathbb{R}^N$ together with a singular term and potentials $a, b: \mathbb{R}^N \to \mathbb{R}$.

\item  We prove the existence of an extreme value $\lambda^* > 0$ such that our main problem has at least two positive solutions for each $\lambda \in ( 0,\lambda^*)$. Hence, we highlight the importance of the size of the parameter $\lambda > 0$ obtaining results on the minimizers in the Nehari sets.

\item We prove a nonexistence result showing that there does not exist weak solutions for our main problem which are also inflection points for the associate fibering map whenever $\lambda = \lambda^*$. This fact allows to prove a multiplicity of solutions for the case $\lambda = \lambda^*$. We show also that the Nehari set is empty for each $\lambda \in (0, \lambda^*)$.

\item By using some fine estimates we shall prove that any minimizer in the Nehari set is a weak solution for our main problem. This is not an easy issue due the fact that our energy functional is only in $C^0$ class.
\end{enumerate}

\subsection{Assumptions and main theorems}
 As mentioned previously, we are concerned with the existence of positive \textit{ground} and \textit{bound states} solutions to the Problem \eqref{P1} involving singular-superlinear nonlinearities with a double weighted nonlocal term. Hence, we need to find a suitable range of parameters to obtain our main results. The main difficulty for this type of problem arises from the non-differentiability of the energy functional. Moreover, we observe also that $\mathcal{N}_{\lambda^*}^0 \neq \emptyset$. In the present work,  we overcome this difficulty by considering the nonlinear Rayleigh quotient showing that there exists an extremal $\lambda^*> 0$ such that the Nehari method can be employed for each $\lambda\in (0,\lambda^*].$ Under these conditions, we define some appropriate minimization problems in order to deal with the lack of regularity of the energy functional. Here we apply some ideas discussed in the recent works \cite{kaye, Yijing}.

Throughout the present work we consider the following assumptions:
\begin{itemize}
    \item[($H_0$)]$0<a\in L^{\frac{2}{2-q}}(\mathbb{R}^N),\;\; b>0,\;\; b \in L^{\infty}(\mathbb{R}^N), \;\;0<q<1<2_{\alpha,\mu}<p<2^*_{\alpha,\mu},~N\geq 3,\;\;\mu>0,\;\;0<2\alpha+\mu< N,\;\; \\ \alpha\in (0,N),~
    2_{\alpha,\mu}= (2N-2\alpha-\mu)/N$  and $2_{\alpha,\mu}^*= (2N-2\alpha-\mu)/(N-2);$
\item[($H_1'$)] $ a\notin L^1(\mathbb R^N),\;\;b\in  L^{\frac{2N}{2N-2\alpha-\mu-(N-2)(p-q)}}(\mathbb R^N);$
\item[($H_2'$)] $ b\in L^{\frac{2N}{2N-2\alpha-\mu-p(N-2)}}(\mathbb R^N).$
\end{itemize}
The potential $ V:\mathbb R^N\to \mathbb R$ is a continuous positive function that satisfies the following statements:
\begin{itemize}
\item[($V_1$)] There exists a constant $V_0>0$ such that
   $ V_0:=\displaystyle\inf_{x\in \mathbb R^N}V(x).$
\item[($V_2$)] For each  $M>0$ we assume that
    $ \left\vert \left\{x\in\mathbb{R}^N\big |V(x)\leq M\right\}\right\vert<\infty.$
\end{itemize}
Now, we shall present some examples for our setting.
\begin{exm}
   Let $a:\mathbb{R}^N\to \mathbb{R}$ be a potential function defined by $a(x)=1/{\left(1+\vert x\vert^2\right)^{\gamma_{3}}}$ where
   $$\frac{N(2-q)}{4}<\gamma_3<\displaystyle\frac{N}{2}.$$
   Using the co-area Theorem we prove that $a\in L^{\frac{2}{2-q}}(\mathbb{R}^N)$ and  $a\notin L^{1}(\R^N)$.
\end{exm}
\begin{exm}
 Consider the potential function $b:\mathbb{R}^N\to \mathbb{R}$ given by $b(x)={1}/{(1+\vert x\vert^2)^{\gamma_4}}$ where $\gamma_4>\max\left\{\displaystyle N/(2\zeta_1),N/(2\zeta_2)\right\}$ and
 $$\zeta_1=\displaystyle\frac{2N}{2N-2\alpha-\mu-(N-2)(p-q)}\qquad\mbox{ and }\qquad\zeta_2=\displaystyle\frac{2N}{2N-2\alpha-\mu-p(N-2)}.$$
 Hence, applying the Co-area Theorem, we observe that
\begin{eqnarray*}
\displaystyle\int_{\R^N}\left(\displaystyle\frac{1}{\left(1+\vert x\vert^2\right)^{\gamma_4}}\right)^\zeta dx \leq c_1+c_2\int_1^\infty \theta^{-2\gamma_4 \zeta+N-1}d\theta<\infty
\end{eqnarray*}
when $\gamma_4>\frac{N}{2\zeta}$. Hence, $b\in L^{\zeta_1}(\R^N)\cap L^{\zeta_2}(\R^N)$.
\end{exm}
Now, we shall provide an example of a potential $V$ that satisfies our hypotheses ($V_1$) and ($V_2$).
\begin{exm}
Consider the coercive potential $V: \mathbb{R}^N \to \mathbb{R}$ given by
 $$V(x) = 1 + \vert x\vert^2.$$
 Fixed $M>0$ we conclude $V(x)=1 + \vert x\vert^2 \le M$ is equivalent to $\vert x \vert \le K = \sqrt{M - 1}$. This implies that $|\{ x \in \mathbb{R}^n; V(x) \le M\}|<\infty$. Therefore, $V$ satisfies the hypotheses $(V_1)$ and $(V_2)$.
\end{exm}

The working space related to our main problem can be defined as follows:
 \begin{eqnarray*}
     X=\left\{ u\in H^1(\mathbb{R}^N)\big|\displaystyle\int_{\mathbb{R}^N}V(x)u^2dx<\infty\right\}.
 \end{eqnarray*}
 Recall that $X$ is a Hilbert space endowed with the inner product given by
 \begin{equation*}\label{prod.interno}
\left<u,\varphi\right>=\displaystyle\int\limits_{\mathbb R^N}(\nabla u\nabla \varphi+V(x)u\varphi) dx,\;\; u,v\in X.
 \end{equation*}
 Furthermore, the norm in $X$ is given by
 $
     \Vert u\Vert^2=\left<u,u\right>,\;\; u\in X.
 $
 The energy functional $J_{\lambda}:X\rightarrow\mathbb R$ associated with Problem \eqref{P1} is defined by
 \begin{equation}\label{funcional energia}
     J_{\lambda}(u)=\displaystyle\frac{1}{2}\Vert u\Vert^2-\frac{\lambda}{q}A(u)-\frac{1}{2p}B(u).
 \end{equation}
It is important to emphasize that the functionals $A,B:X\to \R$ are defined in the following form:
 \begin{equation}\label{AB(u)2}
 A(u)=\displaystyle\int\limits_{\mathbb R^N}a(x)\vert u(x)\vert^{q}dx \qquad\mbox{and}\qquad  B(u)=\displaystyle\int\limits_{\mathbb R^N}\int\limits_{\mathbb R^N}\frac{b(y)\vert u(y)\vert^pb(x)\vert u(x)\vert^p }{\vert x\vert^\alpha\vert x-y\vert^\mu \vert y\vert^\alpha}dxdy,\;\; u\in X.
\end{equation}
Since we are interested in studying the existence and multiplicity of positive solutions, we shall consider the functional $J_{\lambda}$ restricted to the cone of nonnegative functions in $X$. In other words, we are looking for the following space
$$X_+=\{u\in X\backslash\{0\}:u\geq 0\}.$$
Moreover, we also introduce the following notations:
\begin{equation*}\label{HD(u,phi)}
H(u,\varphi)=\displaystyle\int\limits_{\mathbb R^N}a(x)\vert u\vert^{q-2}u\varphi\;dx \qquad\mbox{and}\qquad D(u,\varphi)=\displaystyle\int\limits_{\mathbb R^N}\int\limits_{\mathbb R^N}\frac{b(y)\vert u(y)\vert^p}{\vert x\vert^\alpha\vert x-y\vert^\mu \vert y\vert^\alpha}b(x)\vert u(x)\vert^{p-2}u\varphi\;dx dy, u,\varphi\in X
\end{equation*}
It is important to stress that the functional $H(u,u)$ is well-defined. However, the functional $H$ is not well-defined for all $u,\varphi \in X$ due to the fact that $0<q<1$. Hence, we consider the following definition:
\begin{definition}\label{sol fraca do probLema}
A function $u \in X$ is a weak solution of Problem \eqref{P1} whenever $u>0$ a.e. in $\R^N$, $a u^{q-1}\varphi\in L^{1}(\mathbb{R}^N)$ and
\begin{eqnarray*}\label{eq.sol.fraca}
\displaystyle\int\limits_{\mathbb R^N}(\nabla u\nabla \varphi+V(x)u\varphi) dx= \lambda H(u,\varphi)+ D(u,\varphi), \,\, \varphi \in X.
 \end{eqnarray*}
\end{definition}
Notice that $u\in X \setminus\{0\}$ is a solution of Problem \eqref{P1} if, and only if, $u$ is a critical point of the functional $J_\lambda$.
To state our main results we shall introduce some notations. Define the fibering map $\phi_{\lambda,u}:(0,\infty)\rightarrow \mathbb R$ which is $C^{\infty}$ class given by
\begin{equation*}\label{fibra de J}
\phi_{\lambda,u}(t)=J_{\lambda}(tu)=\displaystyle\frac{t^2}{2}\Vert u\Vert^2-\displaystyle\frac{t^{q}\lambda}{q}A(u)-\displaystyle\frac{t^{2p}}{2p}B(u),\;\; t > 0.
\end{equation*}
Recall that $A(u)$ and $B(u)$ are given by \eqref{AB(u)2}. Hence, for each $ u\in X_+$ and $\lambda>0$, we deduce that
\begin{equation}\label{fibra de J'}
    \phi'_{\lambda,u}(t)=t\Vert u\Vert^2-\lambda t^{q-1}A(u)-t^{2p-1}B(u),
\end{equation}
and
\begin{equation}\label{fibra de J''}
    \phi''_{\lambda,u}(t)=\Vert u\Vert^2-(q-1)\lambda t^{q-2}A(u)-(2p-1)t^{2(p-1)}B(u).
\end{equation}

Consider the Nehari set defined as follows:
\begin{equation}
    \mathcal{N}_{\lambda}=\{u\in X_+\mid\phi_{\lambda,u}'(1)=J'_{\lambda}(u)u=0\}.
\end{equation}
It is standard to split the Nehari set into three disjoint sets as follows:
\begin{eqnarray}\label{N^+2}
  &&\mathcal{N}^+_{\lambda}=\{u\in \mathcal{N}_{\lambda}\;\big |\;\phi_{\lambda,u}''(1)>0\}\nonumber;\\
  \label{N^-2}
&&\mathcal{N}^-_{\lambda}=\{u\in \mathcal{N}_{\lambda}\;\big |\;\phi_{\lambda,u}''(1)<0\};\nonumber\\
\label{N_02}
 &&\mathcal{N}^0_{\lambda}=\{u\in \mathcal{N}_{\lambda}\;\big |\;\phi_{\lambda,u}''(1)=0\}. \nonumber
\end{eqnarray}
Now, we shall consider the study the topological structure of the sets $\mathcal{N_{\lambda}}^+,\;\mathcal{N_{\lambda}}^-\;\hbox{and}\;\;\mathcal{N_{\lambda}}^0$. The main idea is to show the existence of solutions in $\mathcal{N_{\lambda}}^+\;\hbox{and}\;\;\mathcal{N_{\lambda}}^-$. The simplest case occurs assuming that $\mathcal{N_{\lambda}}^0=\emptyset$ holds for some $\lambda > 0$. In this case, we can use the Lagrange Multiplier Theorem showing that any minimizer in the Nehari set is a critical point for the associated energy functional. Hence, the Nehari set is a natural constraint assuming that $\mathcal{N}_\lambda^0$ is empty. Under these conditions, we consider the following minimization problems:
\begin{eqnarray}
\label{c_n+1}  C_{\mathcal{N}_{\lambda}^+}: =\inf \left\{J_{\lambda}(u):u\in\mathcal{N}_{\lambda}^+\right\} & \mbox{and}& \label{C_n-1} C_{\mathcal{N}_{\lambda}^-}:=\inf\left \{J_{\lambda}(u):u\in\mathcal{N}_{\lambda}^-\right\}.
\end{eqnarray}
At this stage, we define the continuous functionals $R_n,R_e : X_+\rightarrow \mathbb R$ associated with the parameter $\lambda > 0$ in the following form:
\begin{eqnarray}\label{R_ne2}
 R_n(u)=\frac{\Vert u\Vert^2-B(u)}{A(u)} \quad \hbox{and}\qquad R_e(u)=\displaystyle q\frac{\displaystyle\frac{1}{2}\Vert u\Vert^2-\displaystyle\frac{1}{2p}B(u)}{
 A(u)}.
\end{eqnarray}
It is not hard to see that
\begin{equation}\label{B5}
u \in\mathcal N_{\lambda}\;\; \hbox{if and only if}\;\; R_n(u)=\lambda,\;u \in X_+,
\end{equation}
\begin{equation}\label{B6}
J_{\lambda}(u)=0\;\; \hbox{if and only if}\;\; R_e(u)=\lambda,\; u \in X_+.
\end{equation}
The basic idea of the Rayleigh quotient method is to understand the behavior of the fibering map $t \mapsto R_n(tu)$ for each $u \in X_+$ fixed. In order to that we define the following extreme:
\begin{eqnarray}\label{lambida^*}
\lambda^*:=\inf_{u \in X_+}\max_{t>0}R_n(tu) \qquad\hbox{and}\quad \lambda_*:=\inf_{u \in X_+}\max_{t>0}R_e(tu).
\end{eqnarray}
Now, we are staying in position to state our main result. Namely, we can prove the following result:
\begin{theorem}\label{teorema 1}
Suppose ($H_0$), ($V_1$), ($V_2$). Then $0<\lambda_*<\lambda^*<\infty$ and for each $\lambda\in (0,\lambda^*)$ the Problem \eqref{P1} admits at least two distinct positive solutions $u_{\lambda},w_\lambda \in H^1(\mathbb R^N)$. Furthermore, we obtain the following statements:
\begin{itemize}
    \item [$a)$] There holds $J''_{\lambda}(u_{\lambda})(u_{\lambda},u_{\lambda})>0$ and $J''_{\lambda}(w_\lambda)(w_\lambda, w_\lambda)<0;$
    \item[$b)$] There exists a constant $c>0$ such that $\Vert w_\lambda\Vert\geq c;$
    \item [$c)$] $u_{\lambda}$ is a ground state solution and $J_{\lambda}(u_{\lambda})<0;$
    \item [$d)$] The functions
 $\lambda\mapsto J_{\lambda}(u_{\lambda})$ and $\lambda\mapsto J_{\lambda}(w_\lambda)$ are decreasing and continuous for $0<\lambda<\lambda^*,$
\item[e)] The solution $w_\lambda$ has the following properties:
\begin{itemize}
\item[$i)$] For each
 $\lambda\in(0,\lambda_*)$ we obtain that $J_{\lambda}(w_\lambda)>0;$
\item[$ii)$] Assume $\lambda=\lambda_*$. Then we obtain that
 $J_{\lambda}(w_\lambda)=0;$
\item[$iii)$] Assume $\lambda\in(\lambda_*,\lambda^*)$. Then $J_{\lambda}(w_\lambda)<0.$
\end{itemize}
\end{itemize}
\end{theorem}

Now, using the previous results, we consider an auxiliary sequence $(\lambda_k)$ such that $\lambda_k \to \lambda^*$ as $k \to \infty$. The main point here is to guarantee that Problem \eqref{P1} does not admit weak solutions in the set $\mathcal{N}_{\lambda^*}^0$. Under these conditions, we are able to prove the following result:

 \begin{theorem}\label{teorema 2}
 Suppose ($H_0$), ($H_1'$), ($H_2'$), ($V_1$) and ($V_2$). Assume also that $\lambda = \lambda^*$. Then Problem \eqref{P1} has at least two positive solutions $w_{\lambda^*}\in \mathcal{N}_{\lambda^*}^-$ and $u_{\lambda^*}\in \mathcal{N}_{\lambda^*}^+.$
\end{theorem}

It is important to stress that  $\mathcal{N}_{\lambda^*}^0$ is nonempty. Therefore, some minimizer $u$ in the Nehari set can be stayed in $\mathcal{N}_\lambda^0$. It is not sufficient for our purposes due the fact that the Lagrange Multiplier Theorem  Problem does not work anymore. Hence, assuming that $u \in \mathcal{N}_\lambda^0$ is a minimizer in the Nehari set, we do not conclude in general that $u$ is a weak solution for the Problem \eqref{P1}. As was told before the main idea is to consider a sequence $(\lambda_k)$ such that $\lambda_k \to \lambda^*$ as $k \to \infty$ where $0 < \lambda_k < \lambda^*$. Therefore, we consider two sequences $(u_k) \in \mathcal{N}_{\lambda_k}^-$ and  $(v_k) \in \mathcal{N}_{\lambda_k}^+$. Under these conditions, we show that Problem \eqref{P1} has at least two solutions using fine estimates and the sequences defined just above.

\subsection{Notations}
In the present work, we shall use the following notations:
\begin{itemize}
    \item Let $f:\R^N\to \R$ be a measurable function, we will denote by $[f\lesseqgtr 0]$ the set $\{x\in\R:f(x)\lesseqgtr 0\}$;
    \item The norm in $L^r(\R^N)$ will denote by $\|.\|_r$ for each $~1\leq r\leq \infty$;
    \item Let us denote by $S_r$ the best constant for the embedding $H^1(\R^N)$;
    \item $r'=r/(r-1)$ denotes the dual number of $r$;
    \item $C_1,C_2,\dots$ denote cumulative constants;
\end{itemize}
\subsection{Outline} This work is organized as follows: In the Section \ref{preli} we present some preliminaries which are fundamental tools to develop our main results. In Section \ref{RQ-rev} we consider some results taking into account the nonlinear Rayleigh quotient proving some technical results. Section \ref{sec-teo1} is devoted to the proof of the Theorem \ref{teorema 1}. In the last section, we prove the Theorem \ref{teorema 2}.

\section{Preliminaries}\label{preli}

Throughout this work, we consider an essential tool in order to get some results. More specifically, we consider the Stein-Weiss inequality and the weighted Hardy-Littlewood-Sobolev inequality as was given in \cite{minbo, Carlos, Beckner, HLSPONDERADA}. Namely, we can state the following result:
\begin{prop}[Stein-Weiss Inequality]\label{Stein-Weiss}
Let $N\geq 3,~s_1,r_1>1,\;\alpha\geq 0,~\mu>0$ such that $0<2\alpha+\mu< N,\;\; f \in L^{r_1}(\mathbb R^N)\;\;\hbox{and}\;\; h \in L^{s_1}(\mathbb R^N),$ where $r_1'$ and $s_1'$ are the conjugate exponents of $r_1>1$ and $s_1>1,$ respectively.
Then, there exists a constant $C=C(\alpha,\mu,N,s_1,r_1)$, such that
\begin{equation}\label{SW}
\displaystyle\int_{\mathbb R^N}\int_
{\mathbb R^N}\frac{f(x)h(y)}{\vert x \vert^{\alpha}\vert x-y \vert^{\mu}\vert y \vert ^{\alpha}}dxdy\leq C(\alpha,\mu,N,s_1,r_1)\Vert f \Vert_{r_1} \Vert h \Vert_{s_1},
\end{equation}
\hbox{where}
\begin{equation}\label{des1}
0\leq\alpha<\min\left\{\displaystyle\frac{N}{r_1'},\displaystyle\frac{N}{s_1'}\right\}\;\;\hbox{and}\;\;\;\frac{1}{r_1}+\frac{1}{s_1}+\frac{2\alpha+\mu}{N}=2.
\end{equation}
\end{prop}
\begin{prop}[Weighted Hardy-Littlewood-Sobolev Inequality]\label{DHLSP}
Let $s_1,r'_1>1,\;\alpha\geq 0,~\mu>0$ such that $0<2\alpha+\mu\leq N$ be fixed. Assume also that $ h \in L^{s_1}(\mathbb R^N)$. Then, there exists a constant $C=C(r_1',s_1,\alpha,N,\mu),$ such that
\begin{equation}\label{HLSP}
\left\Vert \displaystyle\int_{\mathbb R^N}\displaystyle\frac{h(y)}{\vert x\vert^{\alpha}\vert y-x\vert^{\mu}\vert y\vert^{\alpha}}dy\right\Vert_{r_1'}\leq C(r_1',s_1,\alpha,N,\mu)\Vert h\Vert_{s_1},
\end{equation}
with the following relations
\begin{equation}\label{des2}
1+\displaystyle\frac{1}{r_1'}=\displaystyle\frac{2\alpha+\mu}{N}+\displaystyle\frac{1}{s_1}\;\;\hbox{and}\;\;0\leq\alpha<\min\left\{\displaystyle\frac{N}{s_1'},\displaystyle\frac{N}{r_1'}\right\}.
\end{equation}
\end{prop}

According to \cite{BARTSCH_WANG} we obtain the following result:
\begin{lemma}\label{imerso cont e comp}
The subspace $X$ is continuously embedded into $L^q(\mathbb R^N)$ for each $q\in[2,2^*]$. Furthermore, $X$ is compactly embedded into $L^q(\mathbb R^N)$ for each$q\in[2,2^*).$
 \end{lemma}

It is important to mention that the energy functional is well-defined. The key point here is to use the Lemma \ref{imerso cont e comp}. Furthermore, we infer also that  functional $J_\lambda$ is coercive in the Nehari set $\mathcal{N}_{\lambda}$, see for instance \cite{artigo1}. This type of result is crucial in ensuring that any minimizing sequences in the Nehari set are bounded.
It is not hard to prove the following topological properties of the energy functional $J_{\lambda}.$
\begin{lemma}\label{Jsemi contínua}
 Suppose ($H_0$), ($V_1$) and ($V_2$). Then $J_{\lambda}$ is weakly lower semicontinuous.
\end{lemma}
\begin{prop}\label{J continuo1}
 Assume  ($H_0$), ($V_1$) and ($V_2$).Let $J_{\lambda}:X_+\to \mathbb R$ be defined as in \eqref{funcional energia}. Then $J_{\lambda}\in C^0(X_+, \mathbb R)$. More precisely, $A\in C^0(X_+, \mathbb R)$ and $B\in C^1(X_+,\mathbb R)$.
\end{prop}

\section{Generalized nonlinear Rayleigh quotient method}\label{RQ-rev}
In this section, we shall study the nonlinear Rayleigh quotient for the Problem \eqref{P1}. To this end, we will apply some ideas discussed \cite{YAVDAT2017}. The main idea is to find extreme values of parameter $\lambda>0$ such that the Nehari method can be applied.

At this stage, we shall consider the functionals $R_e, R_n$ associated with the parameter $\lambda>0$ defined in \eqref{R_ne2}. Hence, by using \eqref{B5} and \eqref{B6}, we need to understand the behavior of the fibering maps for the functionals $R_n$ and $R_e$. These facts allow us to ensure the existence of a ground-state solution and a bound-state solution for our main problem.
Notice also that \eqref{B5} is crucial in our arguments which shows that there exist two projections in the Nehari sets. Namely, there exists a unique projection in the set $\mathcal{N_{\lambda}}^-$ and another one in the set $\mathcal{N_{\lambda}}^+$ for each $u\in X_+\backslash \{0\}$ where $\lambda \in (0,\lambda^*)$. The statement given in \eqref{B6} is used in order to determine the sign of the energy for solutions provided by the set $\mathcal{N_{\lambda}}^-$. It is not hard to see that $R_e, R_n$ belongs to $C^0(X_+;\mathbb R)$. Under these conditions, we consider the following assertions:
\begin{rmk}\label{relação entre R_n e J'1}
 Suppose ($H_0$), ($V_1$) and ($V_2$). Then, by using \eqref{B5}, we obtain the following statements:
\begin{itemize}
    \item[$i)$] $R_n(u)=\lambda$ if, and only if, $J_{\lambda}'(u)u=0$,
\item[$ii)$] $R_n(u)>\lambda$ if, and only if, $J'_{\lambda}(u)u>0$,
\item[$iii)$] $R_n(u)<\lambda$ if, and only if, $J_{\lambda}'(u)u<0$.
\end{itemize}
\end{rmk}
\begin{rmk}\label{relação entre R_e e J1}
  Suppose ($H_0$), ($V_1$) and ($V_2$). Then, by using \eqref{B6}, we obtain the following statements:
\begin{itemize}
\item[$i)$] $R_e(u)=\lambda$ if, and only if, $J_{\lambda}(u)=0$,
\item[$ii)$] $R_e(u)>\lambda$ if, and only if, $J_{\lambda}(u)>0$,
\item[$iii)$] $R_e(u)<\lambda$ if, and only if, $J_{\lambda}(u)<0$.
\end{itemize}
\end{rmk}
Consider $u\in X_+$, the fibering map related to the $R_n$ for each $t>0$ is given by
\begin{equation*}\label{Q_n2}
    Q_n(t)=R_n(tu)=\frac{t^{2-q}\Vert u\Vert^2-t^{2p-q}B(u)}{A(u)}.
\end{equation*}
Recall that
\begin{equation*}\label{Q'_n1}
    Q'_n(t)=\frac{(2-q)t^{1-q}\Vert u\Vert^2-(2p-q)t^{2p-q-1}B(u)}{A(u)}.
\end{equation*}
The critical point of $t \mapsto Q_n(t)$ allow us to define a function $t_\lambda:X_+\to \mathbb R$ which is given by
\begin{equation*}\label{t-lambda1}
t_n(u)=\left[\frac{(2-q)\Vert u\Vert^2}{(2p-q)B(u)}\right]^{\frac{1}{2p-2}}.
\end{equation*}
Hence, the correspondent critical value is provided by
\begin{equation*}\label{valor critico de Q_n}
Q_n(t_\lambda(u))=\frac{C_{p,q}\Vert u\Vert^{\frac{2p-q}{p-1}}}{A(u)\left[B(u)\right]^{\frac{2-q}{2p-2}}}\quad\mbox{where}\qquad C_{p,q}=\left(\frac{2-q}{2p-q}\right)^{\frac{2-q}{2p-2}}\left(\frac{2p-2}{2p-q}\right).
\end{equation*}
Now, by using $Q_n(t)$ and $Q'_n(t)$, we infer that
$$
\displaystyle\lim_{t\to 0}\frac{Q_n(t)}{t^{2-q}}>0,
\displaystyle\lim_{t\rightarrow \infty}\frac{Q_n(t)}{t^{2p-q}}<0,
\displaystyle\lim_{t\rightarrow 0}\frac{Q'_n(t)}{t^{1-q}}>0,
 \displaystyle\lim_{t\to\infty}\frac{Q'_n(t)}{t^{2p-q-1}}<0.
$$
Since $t_\lambda(u)$ is unique for each $u \in X_+$ we obtain that $Q'_n(t)>0$ for each $t\in (0,t_\lambda(u))$. Moreover, we see that $Q'_n(t)<0$ for each $t>t_\lambda(u)$ and $Q'_n(t)=0$ if, and only if, we have $t = t_\lambda(u)$.
Thus, the number $t_\lambda(u)>0$ is a global maximum point for the fibering map $ t \mapsto Q_n(t)$, i.e., we obtain that $ Q_n(t_\lambda(u))=\displaystyle\max_{t>0}Q_n(t)$.

It is important to stress that $t \mapsto Q_e(t)$ has a similar behavior. More precisely, we consider the function
\begin{equation}\label{Q_e2}
Q_e(t)=R_e(tu)=\displaystyle q\frac{\frac{1}{2}\Vert u\Vert^2t^{2-q}-\displaystyle\frac{t^{2p-q}}{2p}B(u)}{\displaystyle A(u)}.
\end{equation}
Hence, we infer that
\begin{equation*}\label{Q'_e1}
Q'_e(t)=\displaystyle q\frac{\displaystyle\frac{(2-q)}{2}t^{1-q}\Vert u\Vert^2-\displaystyle\frac{(2p-q)}{2p}t^{2p-q-1}B(u)}{\displaystyle A(u)}.
\end{equation*}
Now, the critical point for the function $t \mapsto Q_e(t)$ defines a function  $t_e:X_+\rightarrow \mathbb R$ for each $u \in X_+$. This critical point is given by
\begin{equation*}\label{t_e1}
t_e(u)=\left[\displaystyle\frac{p(2-q)\Vert u\Vert^2}{(2p-q)B(u)}\right]^{\frac{1}{2p-2}}.
\end{equation*}
Thus, the correspondent global maximum value is provide by
\begin{equation*}\label{valor critico de Q_e1}
Q_e(t_e(u))=\displaystyle\frac{\widetilde{C}_{p,q}\Vert u\Vert^{\frac{2p-q}{p-1}}}{A(u)\left[B(u)\right]^{\frac{2-q}{2p-2}}}\qquad\mbox{where}\qquad \widetilde{C}_{p,q}=qp^{\frac{2-q}{2p-2}}\left(\displaystyle\frac{p-1}{2p-q}\right)\left(\displaystyle\frac{2-q}{2p-q}\right)^{\frac{2-q}{2p-2}}.
\end{equation*}
Under these conditions, by using $Q_e(t)$ and $Q'_e(t)$, we deduce that
$$\lim_{t\to 0}\displaystyle\frac{Q_e(t)}{t^{2-q}}>0,~\lim_{t\to \infty}\displaystyle\frac{Q_e(t)}{t^{2p-q}}<0,~
\lim_{t\to 0}\displaystyle\frac{Q'_e(t)}{t^{1-q}}>0~\mbox{and}~\lim_{t\to \infty}\displaystyle\frac{Q'_e(t)}{t^{2p-q-1}}<0.$$
Now, we consider an important relation between the functions $Q_n$ and $Q_e$. More precisely, we obtain the following assertion:
\begin{equation}\label{Q_n-Q_e 2}
Q_n(t)-Q_e(t)=\displaystyle\frac{t}{q}Q'_e(t),~t>0
\end{equation}
As a consequence, by using \eqref{Q_n-Q_e 2}, we deduce the following result:
\begin{prop}\label{Q_n e Q_e 2}
Consider the fibering functions be $Q_n$ and $Q_e$ provided by \eqref{Q_n2} and \eqref{Q_e2}. Then we obtain that
\begin{itemize}
    \item[i)]$ Q'_e(t)>0\;\;\hbox{if, and only if,}\;\; Q_n(t)>Q_e(t)$, for all $t\in (0,t_e(u));$
\item[ii)]$ Q'_e(t)<0\;\;\hbox{if, and only if,}\;\; Q_n(t)<Q_e(t)$, for all $t\in (t_e(u),\infty);$
\item[iii)] $ Q'_e(t)=0\;\;\hbox{if, and only if,}\;\; Q_n(t)=Q_e(t)\;\;\hbox{if, and only if,}\;\; t=t_e(u).$
\end{itemize}
\end{prop}
Now, we consider another important relation between the function $Q'_n(t)$ and the second derivative of the energy functional. In the same way, there exists a relation between $Q'_e(t)$ and the first derivative of the energy functional, see \cite{artigo1}. These results can be summarized as follows:
\begin{equation}\label{d38}
   \displaystyle\frac{d}{dt}R_n(tu)=\displaystyle\frac{q}{t} \displaystyle\frac{J_{\lambda}''(tu)(tu,tu)}{A'(tu)tu}\qquad\mbox{and}\qquad
\frac{d}{dt}R_e(tu)=\displaystyle\frac{q}{t}\displaystyle\frac{J_{\lambda}'(tu)tu}{A(tu)},\qquad t>0.
\end{equation}
As a consequence, by using \eqref{d38}, we infer that
\begin{prop}\label{d/dtR_n(tu) e J''2}
Suppose ($H_0$), ($V_1$) and ($V_2$). Let $u \in X_+$ such that $R_n(tu)=\lambda$ hods for some $t>0$. Then, we obtain the following statements:
\begin{itemize}
\item[$i)$]
 $\displaystyle\frac{d}{dt}R_n(tu)>0\;\; \hbox{if, and only if,}\;\; J_{\lambda}''(tu)(tu,tu)>0$,
\item[$ii)$]$\displaystyle\frac{d}{dt}R_n(tu)<0 \;\; \hbox{if, and only if,}\;\; J_{\lambda}''(tu)(tu,tu)<0$,
\item[$iii)$]$\displaystyle\frac{d}{dt}R_n(tu)=0 \;\; \hbox{if, and only if,}\;\; J_{\lambda}''(tu)(tu,tu)=0$.
\end{itemize}
\end{prop}
\begin{prop}\label{d/dtR_e(tu) e J'2}
Suppose ($H_0$), ($V_1$) and  ($V_2$). Let $u \in X_+,$ such that $R_n(tu)=\lambda$ holds for some $t>0$. Then, we deduce the following statements:
\begin{itemize}
\item[$i)$]
 $\displaystyle\frac{d}{dt}R_e(tu)>0\;\; \hbox{if, and only if,}\;\; J_{\lambda}'(tu)tu>0$,
\item[$ii)$]$\displaystyle\frac{d}{dt}R_e(tu)<0\;\; \hbox{if, and only if,}\;\; J_{\lambda}'(tu)tu<0$,
\item[$iii)$]$\displaystyle\frac{d}{dt}R_e(tu)=0 \;\; \hbox{if, and only if,}\;\; J_{\lambda}'(tu)tu=0$.
\end{itemize}
\end{prop}

From now on, we define the functionals $\Lambda_n,\Lambda_e: X_+\to \mathbb R$ of class $C^0$ given by
\begin{equation}\label{Lambda_n2}
\Lambda_n(u)=R_n(t_\lambda(u)u)=\frac{C_{p,q}\Vert u\Vert^{\frac{2p-q}{p-1}}}{A(u)\left[B(u)\right]^{\frac{2-q}{2p-2}}}\qquad\mbox{and}\qquad
    \Lambda_e(u)=R_e(t_e(u)u)=\frac{\widetilde{C}_{p,q}\Vert u\Vert^{\frac{2p-q}{p-1}}}{A(u)\left[B(u)\right]^{\frac{2-q}{2p-2}}}.
\end{equation}
Notice also that $$\Lambda_e(u)=\displaystyle\frac{\hbox{\~{C}}_{p,q}}{C_{p,q}}\Lambda_n(u)=C\Lambda_n(u).$$ It is not hard to verify that $C\in (0,1)$ which shows that $\hbox{\~{C}}_{p,q}<C_{p,q}$ and $\Lambda_e(u) < \Lambda_n(u)$ holds for each $u \in X_+$.

Now, we shall prove some useful properties for the functional $\Lambda_n$. Namely, we consider the following result:
\begin{lemma}\label{0 homogenea}
Suppose ($H_0$), ($V_1$) and ($V_2$). Then the function $\Lambda_n$ defined in the equation \eqref{Lambda_n2} is continuous, 0-homogeneous, i.e., $\Lambda_n(tu)=\Lambda_n(u)$ for each $u\in X_+$ and $t>0$, weakly lower semicontinuous, unbounded from above, and bounded from below. Furthermore, there exists $u\in X_+,$ such that
\begin{eqnarray}\label{lamb-star}
    \lambda^*=\inf_{v\in X_+}\Lambda_n(v)=\Lambda_n(u)>0.
\end{eqnarray}
\end{lemma}
\begin{proof}
    It is not hard to prove that $\Lambda_n$ is weakly lower semicontinuous and 0-homogeneous. To prove that $\Lambda_n$ is unbounded from above, we consider a sequence $(u_k)\subseteq S:=\{u\in X_+: \|u\|=1\}$ such that $u_k\rightharpoonup 0$ and $u_k\not\rightarrow 0$ in $X$. It follows from Lemma \ref{imerso cont e comp} that $u_k\rightarrow 0$ in $L^t(\mathbb R^N)$ for all $t\in [2,2^*)$. The last assertion implies that $A(u_k) \to 0$ and $B(u_k)\to 0$. As a consequence, we obtain that
    $$ \lim_{k\to \infty}\Lambda_n(u_k)\nonumber=\lim_{k\to \infty}\frac{C_{p,q}\Vert u_k\Vert^{\frac{2p-q}{p-1}}}{A(u_k) B(u_k)^{\frac{2-q}{2p-2}}}=\lim_{k\to \infty}\frac{C_{p,q}}{A(u_k) B(u_k)^{\frac{2-q}{2p-2}}}=\infty.$$
Now, we shall prove that $\Lambda_n$ is bounded from below. Let $u\in S$ be fixed. It follows from Hölder's inequality and Proposition \ref{Stein-Weiss} with  $s=r=\frac{2N}{2N-2\alpha -\mu}$ that
\begin{eqnarray}\label{LI2}
\lambda^*=\inf_{u\in X_+\cap S}\displaystyle\frac{C_{p,q}}{A(u)\left[B(u)\right]^{\frac{2-q}{2p-2}}}
\geq \inf_{u \in X_+\cap S}\frac{C_{p,q}}{\displaystyle\|a\|_{\frac{2}{2-q}}\|u\|^{q}_2\left[B(u)\right]^{\frac{2-q}{2p-2}}}
\geq \inf_{u\in X_+\cap S}\frac{C_1}{\Vert a\Vert_{\frac{2}{2-q}}\Vert u\Vert^q_2\Vert b\vert u\vert^p\Vert_r^{\frac{2-q}{p-1}}}.
\end{eqnarray}
It is important to point out that the condition \eqref{des1} holds. More specifically, we infer that $\frac{\alpha}{N}<1-\frac{1}{r}=\frac{2\alpha+\mu}{2N}$.

Recall also that $b\in L^{\infty}(\mathbb R^N)$. Therefore, by using \eqref{LI2} and Proposition \ref{imerso cont e comp} together with the continuous embedding $H^1(\R^N)\hookrightarrow L^t(\R^N),~t\in [2,2^*)$, we deduce that
\begin{eqnarray*}
    \lambda^*\geq \inf_{u \in X_+\cap S} \displaystyle\frac{C_1}{\Vert a\Vert_{\frac{2}{2-q}}\Vert b\Vert_{\infty}^{\frac{p(2-q)}{p-1}}\Vert u\Vert^q_2\Vert u\Vert_{pr}^{\frac{p(2-q)}{p-1}}}
   \geq  \inf_{u \in X_+\cap S}\frac{C_2}{\Vert a\Vert^{\frac{2}{2-q}}\Vert b\Vert_{\infty}^{\frac{p(2-q)}{p-1}}\Vert u\Vert^q\Vert u\Vert^{\frac{p(2-q)}{p-1}}}
   =\frac{C_2}{\Vert a\Vert^{\frac{2}{2-q}}\Vert b\Vert_{\infty}^{\frac{p(2-q)}{p-1}}}>0.
\end{eqnarray*}
Here, was used the fact that $2<pr<2^*$. Hence, we obtain that $\lambda^*>0$.

From now on, we shall prove that $\lambda^*$ is attained. Let $(u_k)\subset X_+$ be a minimizing sequence for the functional $\Lambda_n$ in $S$. Since $(u_k)$ is bounded there exists $u \in X$ such that $u_k\rightharpoonup u$ in $X$ and $u_k\to u$ in $L^t(\mathbb R^N)$ holds for each $t\in[2,2^*)$. It follows from \eqref{LI2} that $u\neq 0$. Therefore, by using the fact that $\Lambda_n$ is weakly lower semicontinuous, we deduce that
$$\lambda^*\leq \Lambda_n(u)\leq \displaystyle\lim_{k\to \infty}\inf \Lambda_n(u_k)=\lambda^*.$$
Thus, $\Lambda_n(u)=\lambda^*$ for some $u\in X_+.$ This ends the proof.
\end{proof}

\begin{lemma}\label{N0lambda*}
 Suppose ($H_0$), ($V_1$) and ($V_2$).
   Then for each $\lambda>0$ we obtain that $\mathcal N^+_{\lambda},\;\mathcal N^-_{\lambda}\neq \emptyset.$ Furthermore, we deduce that
    \begin{itemize}
        \item [a)]$\mathcal{N}_{\lambda}^0=\emptyset$ for $\lambda\in(0,\lambda^*);$
        \item [b)] $\mathcal{N}_{\lambda}^0\neq\emptyset$ for $\lambda=\lambda^*.$
    \end{itemize}
\end{lemma}
\begin{proof}
Firstly, we prove that $\mathcal{N^+_{\lambda}}\neq\emptyset$ and $\mathcal{N^-_{\lambda}} \neq \emptyset$. According to Lemma \ref{0 homogenea} we infer that $\Lambda_n$ is unbounded from above. Thus, for every $\lambda>0$, there exists $u \in X_+$ such that $\lambda<\Lambda_n(u)$ whenever $0<\lambda<\Lambda_n(u)$. In light of  Proposition \ref{analise dos lambidas} there exists $t_\lambda^+(u)<t_\lambda^-(u)$ such that $t_\lambda^+(u)u\in \mathcal{N^+_{\lambda}}$ and $t_\lambda^-(u)u\in \mathcal{N^-_{\lambda}}$. Therefore, $\mathcal{N^+_{\lambda}}\neq\emptyset$ and $ \mathcal{N^-_{\lambda}} \neq \emptyset$.

It is not hard to prove item $a)$, see for instance \cite{artigo1}. In order to prove the item $b)$ we assume that $\lambda=\lambda^*$. In view of Lemma \ref{0 homogenea} there exists $u\neq 0$ such that $\lambda=\lambda^*=\Lambda_n(u)=Q_n(t_\lambda(u))=R_n(t_\lambda(u)u)$. Then, we obtain that
 $$ R_n(t_\lambda(u)u)=\lambda\qquad\mbox{and}\qquad \frac{d}{dt}R_n(tu)\mid_{t=t_\lambda(u)}=0.$$
Now, by applying Proposition \ref{d/dtR_n(tu) e J''2}, we deduce that $J_{\lambda}''(t_\lambda(u)u)(t_\lambda(u)u,t_\lambda(u)u)=0.$
As a consequence, $t_\lambda(u)u \in \mathcal N_{\lambda}^0$, i.e., the set $\mathcal N_{\lambda}^0$ is nonempty. This finishes the proof.
\end{proof}
The next result proves that any nonzero function in $X_+$ has exactly two projections in the Nehari set. More precisely, one of them in $\mathcal{N}_\lambda^+$ and another projection in $\mathcal{N}_\lambda^-$. In fact, we prove the following result:
\begin{prop}\label{analise dos lambidas}
Suppose ($H_0$), ($V_1$) and ($V_2$). Consider $u\in X_+$ and $\lambda\in (0,\lambda^*)$. Then the fibering map $\phi_{\lambda,u}(t)=J_{\lambda}(tu)$ has exactly two distinct critical points $0<t_\lambda^+(u)<t_\lambda(u)<t_\lambda^-(u)$. Furthermore, the number $t_\lambda^+(u)$ is a local minimum point for the fibering map $\phi_{\lambda,u}$ which satisfies $t_\lambda^+(u)u \in \mathcal {N}_{\lambda}^+$ and the number $t_\lambda^-(u)$ is a local maximum for $\phi_{\lambda,u}$ which satisfies $t_\lambda^-(u)u \in \mathcal {N}_{\lambda}^-$.
\end{prop}
\begin{proof}
The proof follows the same ideas employed in the proof of \cite[Prposition 3.10]{artigo1}. We omit the details.
\end{proof}
Now, by using Lemma \ref{N0lambda*}, we consider the following result:
\begin{prop}
 Suppose ($H_0$), ($V_1$) and ($V_2$). Assume that $\lambda=\lambda^*$ and $u\in X_+$ such that $u$ is a minimizer for the functional $\Lambda_n.$ Then the fibering map $\phi_{\lambda,u}(t)=J_{\lambda}(tu)$ has a unique critical point and $\phi_{\lambda,u}(t)$ is decreasing for all $t>0$.
\end{prop}
\begin{proof}
Recall that $R_n(tu)\leq \lambda^*,~t>0$ holds where the equality holds if, and only if $t=t_\lambda(u)$. It follows from Remark \ref{relação entre R_n e J'1} $i)$ and $iii)$ that $\phi_{\lambda,u}'(t)=J'_{\lambda}(tu)u\leq 0,~t>0.$ Consequently, we infer that $\phi_{\lambda,u}(t)=J_{\lambda}(tu)$ has a unique critical point $t_\lambda(u)$. Moreover, $\phi_{\lambda,u}(t)$ is decreasing for all $t>0$. This ends the proof.
\end{proof}
Now, we shall consider a result whose the proof is based on in \cite[Lemma 2.5]{kaye}. Namely, we prove the following result:
\begin{lemma}\label{3 passos}
 Suppose ($H_0$), ($V_1$) and ($V_2$). Consider the function $v\in X\backslash \{0\},$ such that $\Lambda_n(v)=\lambda^*$ and define the function $u=t_\lambda(v)v$. Then $u$ is a weak solution to the following nonlocal elliptic problem:
    \begin{equation}\label{P3}
\left\{\begin{array}{ll}
   -2\Delta u+2V(x)u=q\lambda^*a(x)\vert u\vert^{q-2}u+2p\displaystyle\int_{\mathbb R^N}\displaystyle\frac{b(y)\vert u(y)\vert^p b(x)\vert u\vert^{p-2}u}{\vert x\vert^{\alpha}\vert x-y\vert^{\mu}\vert y\vert^{\alpha}}dy,~\hbox{in}\;\; \mathbb R^N,\\
   u \in H^1(\mathbb R^N).
    \end{array}
\right.  \tag{$Q_{\lambda^*}$}
\end{equation}
\begin{proof}
It follows from Proposition \ref{0 homogenea} that $v$ is a global minimizer for the fibering map $\phi_{\lambda,v}$. Recall also that $\Lambda_n$ is 0-homogeneous. Thus, $\Lambda_n(v)=\Lambda_n(t_\lambda(v)v)=\lambda^*$. Now, we shall prove that $u:=t_\lambda(v)v\in X\backslash \{0\}$ satisfies $\Lambda^{'}_n(u) \varphi =0$ for any $\varphi \in X$, i.e., $u$ is a weak solution to the Problem \eqref{P3}.
Indeed, we mention that $\Lambda_n$ can be written as
\begin{equation}\label{lambda=fg}
   \Lambda_n(u)=C(p,q)f(u)g(u),\qquad\mbox{where}\qquad f(u)=\displaystyle\frac{1}{A(u)}\qquad\mbox{and}\qquad g(u)=\displaystyle\frac{\Vert u\Vert^{\frac{2p-q}{p-1}}}{[B(u)]^{\frac{2-q}{2(p-1)}}}.
\end{equation}
Under these conditions, we shall split the proof into three steps. Initially, we consider the following steps:

\noindent Step $1)$. In this step we shall prove that $\left<f'(u),\varphi\right>$ exists for all $\varphi>0$. In order to prove that, for any $\varphi>0, \varphi \in X$, it follows by continuity of $B$ that $g(u+t\varphi)$ is well-defined for each $t>0$. Furthermore, we deduce that $$\left<g'(u),\varphi\right>=\displaystyle\lim_{t\to 0}\displaystyle\frac{g(u+t\varphi)-g(u)}{t}.$$
Notice also that $u$ is a minimum point for the functional $\Lambda_n(u)$. Therefore, we obtain that
    \begin{equation}\label{lambida-lambida*}
    \Lambda_n(u+t\varphi)-\Lambda_n(u)=\Lambda_n(u+t\varphi)-\lambda^*\geq 0,
    \end{equation}
holds for all $t>0$ small enough. Now, by using \eqref{lambda=fg} and \eqref{lambida-lambida*}, we infer that
\begin{eqnarray*}
  0\nonumber&\leq&\Lambda_n(u+t\varphi)-\Lambda_n(u)=f(u+t\varphi)g(u+t\varphi)-f(u)g(u) =\left[g(u+t\varphi)-g(u)\right]f(u+t\varphi)+\left[f(u+t\varphi)-f(u)\right]g(u).
    \end{eqnarray*}
Hence, we deduce that
\begin{equation}\label{g(u+tfi)-g(u)}
        \left[g(u+t\varphi)-g(u)\right]f(u+t\varphi)\geq -g(u)\left[f(u+t\varphi)-f(u)\right].
\end{equation}
Now, we  define the function $M(t)=f(u+t\varphi), t > 0$. In light of the Mean Value Theorem there exists $0<\theta=\theta(t)< t<1$ such that
    $$
   \displaystyle\frac{M(t)-M(0)}{t}=M'(\theta)=-[A(u+\theta\varphi)]^{-2}\;\displaystyle\frac{d}{dt}A(u+\theta\varphi)=q\int_\Omega a(x)(u+\theta \varphi)^{q-1}\varphi dx.
   $$
It is important to stress that the last integral is well-defined due to the fact that $0\leq a(u+\theta \varphi)^{q-1}\varphi\leq a\varphi^q$ a.e. in $\R^N$.
As a consequence, by using \eqref{g(u+tfi)-g(u)}, we have that
\begin{eqnarray}\label{g1}
    \infty \nonumber& > \nonumber& \displaystyle\liminf_{t\to 0}\displaystyle\frac{g(u+t\varphi)-g(u)}{t}f(u+t\varphi)~\geq~  \displaystyle\liminf_{t\to 0}-g(u)M'(\theta)\\
    &=& g(u)\displaystyle\liminf_{t\to 0}[A(u+\theta\varphi)]^{-2}q\int_\Omega a(x) (u+\theta \varphi)^{q-1}\varphi dx.
  \end{eqnarray}
According to Fatou's Lemma and \eqref{g1} we obtain that
\begin{eqnarray}\label{positive}
\infty> g'(u)\varphi f(u) \geq g(u)\left[\displaystyle\int_{\mathbb R^N}a(x)\vert u\vert^qdx\right]^{-2}q\displaystyle\int_{\mathbb R^N}a(x)G(x)\varphi dx,
\end{eqnarray}
where
\begin{eqnarray*}
G(x)=\left\{\begin{array}{ll}
   \vert u(x)\vert^{q-2}u(x), \;\;\hbox{if}\;\; u(x)\neq
0,\\
\infty,\;\;\hbox{if}\;\;u(x)=0.
\end{array}
\right.
\end{eqnarray*}
Now, by applying \eqref{positive}, we deduce that $u>0$ a.e. in $\R^N$ and $a u^{q-1}\varphi\in L^1(\R^N)$ holds for each $\varphi>0$. Hence, we obtain that
$$\left<A'(u),\varphi\right>=q\displaystyle\int_{\mathbb R^N}a(x)\vert u\vert^{q-2}u\varphi dx, \,\, \varphi\in X_+.$$
Recall also that $f(u)=\left[A(u)\right]^{-1}$. Under these conditions, we obtain that
$$\left<f'(u),\varphi\right>=-\left[\displaystyle\int_{\mathbb R^N}a(x)\vert u\vert^q dx\right]^{-2}q\displaystyle\int_{\mathbb R^N}a(x)\vert u\vert^{q-2}u\varphi dx,\;\;\varphi\in X_+.$$
\noindent Step $2)$. In this step we shall prove that
\begin{eqnarray}
   \nonumber 2\left<u, \varphi\right>-q\lambda^*\displaystyle\int_{\mathbb R^N}a(x)\vert u\vert^{q-2}u\varphi dx-2p\displaystyle\int_{\mathbb R^N}\displaystyle\int_{\mathbb R^N}\displaystyle\frac{b(y)\vert u\vert^pb(x)\vert u\vert^{p-2}u\varphi dx}{\vert x\vert^{\alpha}\vert x-y\vert^{\mu}\vert y\vert^{\alpha}}\geq 0, ~\varphi\in X_+.
\end{eqnarray}
Once again, by using the fact that $u\in X$ is a minimum point for the functional$\Lambda_n,$ it follows from the previous step that
$$
\Lambda'_n(u)\varphi=\displaystyle\lim_{t\to 0}\displaystyle\frac{\Lambda_n(u+t\varphi)-\Lambda_n(u)}{t}\geq 0.
$$
Under these conditions, we mention that
\begin{eqnarray}\label{d10}
   0&\leq &\Lambda'_n(u)\varphi \nonumber=        C_{p,q}\displaystyle\frac{\displaystyle\frac{2p-q}{p-1}\Vert u\Vert^{\frac{p-q+2}{p-1}}\left<u,\varphi\right>A(u)\left[B(u)\right]^{\frac{2-q}{2p-2}}}{\left[A(u)\right]^2\left[B(u)\right]^{\frac{2-q}{p-1}}}-\displaystyle\frac{ C_{p,q}\Vert u\Vert^{\frac{2p-q}{p-1}}}{\left[A(u)\right]^2\left[B(u)\right]^{\frac{2-q}{p-1}}}\left[q\left[B(u)\right]^{\frac{2-q}{2p-2}}\displaystyle H(u,\varphi)\right]\\
    &-&\displaystyle\frac{ C_{p,q}\Vert u\Vert^{\frac{2p-q}{p-1}}}{\left[A(u)\right]^2\left[B(u)\right]^{\frac{2-q}{p-1}}}\left[A(u)\left(\displaystyle\frac{2-q}{2p-2}\right)\left[B(u)\right]^{\frac{4-q-2p}{2p-2}}2p D(u,\varphi)\right].
\end{eqnarray}
where $H(u,\varphi)$ and $D(u,\varphi)$ were defined in \eqref{HD(u,phi)}. Without loss of generality, we assume that $\Vert u\Vert=1$. It follows from \eqref{d10} that
\begin{eqnarray}\label{d11}
    0&\leq & \displaystyle\frac{\left(\displaystyle\frac{2p-q}{p-1}\right)\left[B(u)\right]^{\frac{2-q}{2p-2}}\left<u, \varphi\right>-\left(\displaystyle\frac{2-q}{2p-2}\right)\left[B(u)\right]^{\frac{4-q-2p}{2p-2}}2p D(u,\varphi)}{A(u)\left[B(u)\right]^{\frac{2-q}{p-1}}}- \displaystyle\frac{q\left[B(u)\right]^{\frac{2-q}{2p-2}}\displaystyle H(u,\varphi)}{\left[A(u)\right]^2\left[B(u)\right]^{\frac{2-q}{p-1}}},
\end{eqnarray}
holds for all $\varphi\in X_+$. Therefore, we obtain that $u\in \mathcal{N}_{\lambda^*}^0$. In virtue of \eqref{fibra de J'} and \eqref{fibra de J''} we see that
\begin{equation}\label{d12}
    0=\phi_{\lambda,u}'(1)=1-\lambda^*A(u)-B(u), \,\,
0=\phi_{\lambda,u}''(1)=1-(q-1)\lambda^*A(u)-(2p-1)B(u).
\end{equation}
As a consequence, by using \eqref{d12}, we deduce that
\begin{equation}\label{d14}
    A(u)=\displaystyle\frac{2p-2}{\lambda^*(2p-q)}.
\end{equation}
Furthermore, by using \eqref{d14} and \eqref{d12}, we verify that
\begin{equation}\label{d15a}
    B(u)=\displaystyle\frac{2-q}{2p-q}.
\end{equation}
According to \eqref{d14}, \eqref{d15a} and \eqref{d11} we deduce that
$$
0\leq \displaystyle\frac{(2p-q)^{\frac{4p-2-q}{2p-2}}\lambda^*}{4(p-1)^2(2-q)^{\frac{2-q}{2p-2}}}\left[2\left<u,\varphi\right>-2p D(u,\varphi)-\lambda^*q\displaystyle\int_{\mathbb R^N}a(x)\vert u\vert^{q-2}u\varphi dx\right].
$$
Therefore, we obtain
\begin{equation}\label{d15}
    2\left<u,\varphi\right>-2p D(u,\varphi)-\lambda^*qH(u,\varphi)\geq 0,\,\, \varphi\in X_+.
\end{equation}
\noindent Step $3)$. In this step we shall prove that
$$
2\left<u,\varphi\right>-2pD(u,\varphi)-\lambda^*qH(u,\varphi)=0,\, \,  \varphi\in X.
$$
Let us consider $\Theta=(u+\varepsilon\varphi)^+\in X_+$ for any $\varepsilon>0.$
Since $\Theta\in X_+$ it follows from that \eqref{d15} is satisfied. In particular, we have that
$$
2\left<u,\Theta\right>-\lambda^*q\displaystyle\int_{\mathbb R^N}a(x)\vert u\vert^{q-2}u\Theta dx-2p\displaystyle\int_{\mathbb R^N}\displaystyle\int_{\mathbb R^N}\displaystyle\frac{b(y)\vert u\vert^pb(x)\vert u\vert^{p-2}u\Theta}{\vert x\vert^{\alpha}\vert x-y\vert^{\mu}\vert y\vert^{\alpha}}dxdy\geq 0.
$$
Recall also that $(u+\varepsilon\varphi)^+=(u+\varepsilon\varphi)-(u+\varepsilon\varphi)^-$. Hence, we obtain
\begin{eqnarray}\label{d16}
    0 \nonumber&\leq & 2\left<u,u+\varepsilon\varphi\right>-\lambda^*q\displaystyle\int_{\mathbb R^N} a(x)\vert u\vert^{q-2}u(u+\varepsilon\varphi)dx-  2p\displaystyle\int_{\mathbb R^N}\displaystyle\int_{\mathbb R^N}\displaystyle\frac{b(y)\vert u\vert^pb(x)\vert u\vert^{p-2}u(u+\varepsilon\varphi)}{\vert x\vert^{\alpha}\vert x-y\vert^{\mu}\vert y\vert^{\alpha}}dxdy\\
   \nonumber &-& 2\left<u,(u+\varepsilon\varphi)^-\right>+\lambda^*q\displaystyle\int_{[u+\varepsilon\varphi\leq 0]}a(x)\vert u\vert^{q-2}u(u+\varepsilon\varphi)dx+ 2p\displaystyle\int_{[u+\varepsilon\varphi\leq 0]}\displaystyle\int_{\mathbb R^N}\displaystyle\frac{b(y)\vert u\vert^pb(x)\vert u\vert^{p-2}u(u+\varepsilon\varphi)}{\vert x\vert^{\alpha}\vert x-y\vert^{\mu}\vert y\vert^{\alpha}}dxdy\\
   \nonumber&\leq & 2 \Vert u\Vert^2-\lambda^*q\displaystyle\int_{\mathbb R^N}a(x)\vert u\vert^q dx-2p\displaystyle\int_{\mathbb R^N}\displaystyle\int_{\mathbb R^N}\displaystyle\frac{b(y)\vert u\vert^p b(x)\vert u\vert^p}{\vert x\vert^{\alpha}\vert x-y\vert^{\mu}\vert y\vert^{\alpha}}dxdy+ 2\varepsilon\left<u, \varphi\right>-\lambda^*q\varepsilon\displaystyle\int_{\mathbb R^N}a(x)\vert u\vert^{q-2}u\varphi dx\\
   &-& 2\varepsilon\displaystyle\int\limits_{\{u+\varepsilon\varphi\leq 0\}}\nabla u\nabla\varphi+V(x)u\varphi dx- 2p\varepsilon\displaystyle\int_{\mathbb R^N}\displaystyle\int_{\mathbb R^N}\displaystyle\frac{b(y)\vert u\vert^p b(x)\vert u\vert^{p-2}u\varphi}{\vert x\vert^{\alpha}\vert x-y\vert^{\mu}\vert y\vert^{\alpha}}dxdy.
\end{eqnarray}
Notice also that $\Lambda_n(u)=\lambda^*$. Under these conditions, we mention that
\begin{equation}\label{d17}
    0=\phi_{\lambda^*,\mu}'(1)=\Vert u\Vert^2-\lambda^*A(u)-B(u), \,\,
    0=\phi_{\lambda^*,\mu}''(1)=\Vert u\Vert^2-(q-1)\lambda^*A(u)-(2p-1)B(u).
\end{equation}
Therefore, by using \eqref{d17}, we deduce that
\begin{equation}\label{d19}
0=2\Vert u\Vert^2-q\lambda^*A(u)-2pB(u).
\end{equation}
Now, by using \eqref{d19} and \eqref{d16}, we infer that
\begin{eqnarray}\label{d21}
    0 &\leq & 2\left<u,\varphi\right>-\lambda^*q\displaystyle\int_{\mathbb R^N}a(x)\vert u\vert^{q-2}u\varphi dx- 2p D(u,\varphi)-2\displaystyle\int_{[u+\varepsilon\varphi\leq 0]}\nabla u\nabla\varphi+V(x) u\varphi dx.
\end{eqnarray}
It is important to emphasize that
\begin{eqnarray*}\label{TCD-grad}
\left \vert\vert\nabla u\nabla\varphi+V(x)u\varphi\vert\chi_{\{u+\varepsilon\varphi\leq 0\}}(x)\right\vert\leq \vert \nabla u\nabla\varphi+V(x)u\varphi\vert\vert\chi_{\{u+\varepsilon\varphi\leq 0\}}(x)\vert\leq\vert \nabla u\nabla\varphi+V(x)u\varphi\vert\in L^1(\mathbb R^N).
\end{eqnarray*}
Thus, by using the Dominated Convergence Theorem, we conclude that
\begin{equation*}\label{d22}
 \displaystyle\lim_{\varepsilon\to 0}\displaystyle\int_{\mathbb R^N}\vert\nabla u\nabla\varphi+V(x) u\varphi\vert \chi_{\{u+\varepsilon\varphi\leq 0\}}(x)dx= 0.
\end{equation*}
Therefore, by using \eqref{d21} and doing $\varepsilon \to 0$, we infer that
\begin{equation*}\label{sol.fraca>=0}
0\leq 2\left<u,\varphi\right>-\lambda^*q\displaystyle\int_{\mathbb R^N}a(x)\vert u\vert^{q-2}u\varphi dx -2pD(u,\varphi),\, \, \varphi\in X.
\end{equation*}
As a consequence, $u$ is a weak solution for the Problem \eqref{P3}. This ends the proof.
\end{proof}
\end{lemma}
As a fundamental ingredient in order to show the existence and multiplicity of solutions for the Problem \eqref{P1} we need to ensure the continuity and monotonicity of the energy functional restricted to Nehari sets $\mathcal{N}_{\lambda}^+$ and $\mathcal{N}_{\lambda}^-$. Let us define
$J^+_{\lambda}:X_+ \setminus \{0\} \to \mathbb R$ and $J^-_{\lambda}: X_+ \setminus \{0\}\to \mathbb R$ given by
$
    J^+_{\lambda}(u)=J_{\lambda}(t_{\lambda}^+(u)u)\;\;\hbox{and}\;\;J^-_{\lambda}(u)=J_{\lambda}(t_{\lambda}^-(u)u).
$
\begin{lemma}\label{t cont. em relação a lambda}
 Suppose ($H_0$), ($V_1$) and ($V_2$). Let $u\in X_+$ and $I\subset \mathbb R$ an open interval such that $t_{\lambda}^{+}(u)$ and $t^-_{\lambda}(u)$ are well-defined for all $\lambda\in I.$ Therefore, we obtain the following assertions:
\begin{itemize}
    \item [$a)$] The functions $I\ni\lambda\mapsto t^{\pm}_{\lambda}(u)$ are $C^1.$ Furthermore, $I\ni\lambda\mapsto t^+_{\lambda}(u)$ is increasing while $I\ni\lambda\mapsto t^-_{\lambda}(u)$ is decreasing.
    \item[$b)$] It holds that $I\ni\lambda\mapsto J^{\pm}_{\lambda}(u)$ define functions of $C^1$ class which are also decreasing.
\end{itemize}
\begin{proof}
   $a)$ In order to prove that $I\ni\lambda\mapsto t^{\pm}_{\lambda}(u)$ are in $C^1$ class we define an auxiliary function $F: (0, \infty) \times (0, \infty) \times X_+ \to \mathbb{R}$ of class $C^1$ given by
     $$F(\lambda,t,u)=\phi_{\lambda}'(t)=t\Vert u\Vert^2-\lambda t^{q-1}A(u)-t^{2p-1}B(u),\;\;t>0.$$
   For each $\lambda_i\in I,$ we have that $t^+_{\lambda_i}(u)$ is well-defined and $t^+_{\lambda_i}(u)u\in\mathcal{N}^+_{\lambda_i}.$ We observe that
    $$F(\lambda_i,t_{\lambda_i}^+(u),u)=\phi_{\lambda_i,u}'(t^+_{\lambda_i}(u))=0\qquad\mbox{and}\qquad\displaystyle\frac{\partial F}{\partial t}(\lambda_i,t,u){\big|}_{t=t_{\lambda_i}^+(u)}=\phi_{\lambda_i,u}''(t_{\lambda_i}^+(u))>0.$$
Now, by using the Implicit Function Theorem \cite[Remark 4.2.3, Theorem 4.2.1]{DRABEK}, we deduce that $t_{\lambda}^+(u)\in C^1\left((\lambda_i-\varepsilon,\lambda_i+\varepsilon),\mathbb R\right)$ holds for any $\varepsilon>0$ small enough.  Notice that $I$ is an open interval and $\lambda_i$ is arbitrary. The last assertion yields $t_{\lambda}^+(u)\in C^1(I, \mathbb R).$
Moreover, we mention that
$$F(\lambda,t^+_{\lambda}(u),u)=0, \qquad \frac{\partial F}{\partial t}(\lambda,t,u){\big|}_{t=t_{\lambda}^+(u)}>0\qquad\mbox{and}\qquad\displaystyle\frac{\partial F}{\partial \lambda}(\lambda, t_{\lambda}^+(u),u)+\displaystyle\frac{\partial F}{\partial t}(\lambda,t,u){\big|}_{t=t_{\lambda}^+(u)}\frac{dt^+_{\lambda}(u)}{d\lambda}=0.$$
In particular, we see that
$$\displaystyle\frac{dt^+_{\lambda}(u)}{d\lambda}=-\displaystyle\frac{\displaystyle\frac{\partial F}{\partial\lambda}(\lambda,t^+_{\lambda}(u),u)}{\frac{\partial F}{\partial t}(\lambda,t,u){\big|}_{t=t_{\lambda}^+(u)}}=\displaystyle\frac{t^+_{\lambda}(u)^{q-1}A(u)}{\phi''_{\lambda,u}(t^+_{\lambda}(u))}>0.$$
Thus, the function $\lambda \mapsto t^+_{\lambda}(u)$ is increasing. Similarly, we prove that $\lambda \mapsto t^-_{\lambda}(u)$ is decreasing which defines also a function in $C^1(I, \mathbb{R})$.\\
$b)$ Notice that $t^+_{\lambda}(u) > 0$ and $J^+_{\lambda}(u) = J_{\lambda}(t^+_{\lambda}(u)u)$. It follows from part $a)$ that $\lambda \mapsto J^+_{\lambda}(u)$ is $C^1$. It is important to mention that
$
    J^+_{\lambda}(u)=\phi_{\lambda,u}(t^+_{\lambda}(u)).
$
Under these conditions, we have that
\begin{eqnarray*}
&&\displaystyle\frac{dJ^+_{\lambda}(u)}{d\lambda}=\displaystyle\frac{\partial \phi_{\lambda,u}(t^+_{\lambda}(u))}{\partial t}\displaystyle\frac{dt^+_{\lambda}(u)}{d\lambda}+\displaystyle\frac{\partial \phi_{\lambda,u}(t_{\lambda}^+(u))}{\partial\lambda}=\displaystyle\frac{\partial \phi_{\lambda,u}(t_{\lambda}^+(u))}{\partial\lambda}
= -\displaystyle\frac{t^+_{\lambda}(u)^q A(u)}{q}<0.
\end{eqnarray*}
Here was used the fact that $t^+_{\lambda}(u)u\in \mathcal{N}^+_{\lambda}.$ Now, using the last estimate, we deduce that $I\ni\lambda\mapsto J^+_{\lambda}(u)$ is decreasing. Similarly, we prove that $I\ni\lambda\mapsto J^-_{\lambda}(u)$ is a continuous and decreasing function. This finishes the proof.
\end{proof}
\end{lemma}
Now, we are able to prove that the functions $\lambda \mapsto J_{\lambda}^{\pm}$ are continuous for each $\lambda \in (0,\lambda^*]$. Namely, we shall prove the following result:
\begin{cor}
Suppose ($H_0$), ($H'_1$), ($H'_2$), ($V_1$), ($V_2$). Consider a fixed function $u\in X_+$. Then we obtain the following identities:
$$\displaystyle\lim_{\lambda\to \lambda^*}t^-_{\lambda}(u)=t_{\lambda^*}^-(u),\;\;\displaystyle\lim_{\lambda\to \lambda^*}t^+_{\lambda}(u)=t_{\lambda^*}^+(u), \,\, \mbox{and} \,\, \displaystyle\lim_{\lambda\to \lambda^*}J^-_{\lambda}(u)=J_{\lambda^*}(t^-_{\lambda^*}(u)u),\;\;\displaystyle\lim_{\lambda\to \lambda^*}J^+_{\lambda}(u)=J_{\lambda^*}(t^+_{\lambda^*}(u)u).$$
\end{cor}
\begin{proof}
The proof follows using a standard argument together with Proposition \ref{t cont. em relação a lambda}. The details are omitted.
\end{proof}
Now, we shall prove the following useful result:
\begin{lemma}\label{C_n+<01}
Suppose ($H_0$), ($V_1$), ($V_2$) and $\lambda\in (0,\lambda^*)$. Then $C_{\mathcal{N}_{\lambda}^+}=J_{\lambda}(u)<0.$
\end{lemma}
\begin{proof}
It is easy to see that $C_{\mathcal{N}_{\lambda}^+}=\displaystyle\inf_{v\in\mathcal{N}_{\lambda}^+}J_{\lambda}(v)\leq J_{\lambda}(tv)<0$ holds for all $v\in X_+$ and $t>0$ small enough. This ends the proof.
\end{proof}

From now on, we shall prove that minimizers sequences strongly converge for some point $u \in X_+$ where $u \neq 0$. Firstly, we prove the following result:
\begin{lemma}\label{J(u)=C_ N+1}
  Suppose ($H_0$), ($V_1$), ($V_2$) and $\lambda\in (0,\lambda^*)$. Consider  $(u_k)\subseteq\mathcal{N}_{\lambda}^+$ a minimizing sequence for the functional $ J_{\lambda}$ in $\mathcal{N}_{\lambda}^+.$ Then there exists $ u_\lambda \in \mathcal{N}_\lambda^+$ such that, up to a subsequence, $u_k \to u_\lambda$ in  $X$. Moreover, we obtain that $C_{\mathcal{N}_{\lambda}^+}=J_{\lambda}(u_\lambda)$.
\end{lemma}
\begin{proof}
Initially, we consider a minimizing sequence $(u_k)\subset\mathcal{N}_{\lambda}^+$, i.e., $$J_{\lambda}(u_k) \to C_{\mathcal{N}_{\lambda}^+}=\displaystyle\inf_{w\in \mathcal{N}_{\lambda}^+}J_{\lambda}(w).$$
Since $J_{\lambda}$ is coercive in the Nehari set $\mathcal{N}_{\lambda}$ we know that $(u_k)$ is bounded in $X$. Thus, up to a subsequence, there exists $u\in X_+$ such that $u_k \rightharpoonup u$ in $X$. Now, we claim that $u\neq 0$. Otherwise, we obtain that $u=0$ and $u_k\rightharpoonup 0$ in $X$. Hence, $A(u_k)\to 0$ and $B(u_k)\to 0$. Therefore, by using \eqref{fibra de J'}, we obtain that
$0=J'_{\lambda}(u_k)u_k=\Vert u_k\Vert^2-\lambda A(u_k)-B(u_k).$
As a consequence, $\Vert u_k\Vert\to 0$ and $J_{\lambda}(u_k)\to 0$ as $k \to \infty$.
This is a contradiction due to the fact that $$J_{\lambda}(u_k)\to C_{\mathcal{N}_{\lambda}^+}<0.$$
Thus, we obtain that $u\neq 0$. Now, the proof of the strong convergence follows arguing by contradiction. Let us suppose that $u_k \rightharpoonup u$ and $u_k \not\rightarrow u$ in $X$. These facts implies that
$$\Vert u\Vert<\displaystyle\lim_{k\to \infty}\inf\Vert u_k\Vert.$$
Consider also the fibering map $\phi_{\lambda,u}:[0,\infty) \to \mathbb R$ given by $\phi_{\lambda,u}(t)=J_{\lambda}(tu),~t\geq 0,~u\in X$.
According to Proposition \ref{analise dos lambidas} there exists a unique $t_\lambda^+(u)>0$ such that $t_\lambda^+(u)u \in \mathcal{N}_{\lambda}^+$. Moreover, we know that $\phi'_{\lambda,u}(t_\lambda^+(u))=J_{\lambda}'(t_\lambda^+(u)u)u=0.$ In particular, by using the fact that $u \mapsto J'_\lambda(u) u$ is weakly lower semicontinous, we obtain that
\begin{eqnarray*}
    0=t_\lambda^+(u)\displaystyle\frac{d}{dt}J_{\lambda}(tu)\big|_{t=t_\lambda^+(u)}=J'(t_\lambda^+(u)u)t_\lambda^+(u)u<\displaystyle\liminf_{k\to\infty} J_{\lambda}'(t_\lambda^+(u)u_k)t_\lambda^+(u)u_k.
\end{eqnarray*}
As a consequence, we obtain that
$
J_{\lambda}'(t_\lambda^+(u)u_k)t_\lambda^+(u)u_k>0$
hold for all $k$ large enough. Therefore, $t_\lambda^+(u_k)<t_\lambda^+(u)<t_\lambda^-(u_k)$. Notice also that $t_\lambda^+(u_k)=1$ which implies that $1\leq t_\lambda^+(u)$. Now, by using the fact that $t_\lambda^+(u)u\in \mathcal{N}_{\lambda}^+$ and $t \mapsto J_\lambda(t u_k)$ is decreasing in the interval $[0, t_n^+(u_k)]$, we deduce that
\begin{equation*}
    C_{\mathcal{N}^+_{\lambda}}\leq J_{\lambda}(t_\lambda^+(u)u) \leq J_{\lambda}(u)<\displaystyle\lim_{k\to \infty}J_{\lambda}(u_k)=C_{\mathcal{N}_{\lambda}^+}.
\end{equation*}
This is a contradiction proving that $u_k\to u$ in $X$. Since $J_{\lambda}$ is in $ C^0$ class we have that $J_{\lambda}(u_k)\to J_{\lambda}(u)=C_{\mathcal{N}_{\lambda}^+}$. This concludes the proof.
\end{proof}

\begin{lemma}\label{N^- longe do zero}
Suppose ($H_0$), ($V_1$), ($V_2$) and $\lambda\in (0,\lambda^*)$. Then, for each $w\in \mathcal{N}^-_{\lambda}\cup \mathcal{N}^0_{\lambda}$, there exists a constant $0<C=C(N,p,q)$ that does not depend on $\lambda$ such that $\Vert w\Vert\geq C$. In particular, the sets $\mathcal{N}_\lambda^-$ and $\mathcal{N}_\lambda^0$ are closed.
\end{lemma}
\begin{proof}
 Let us consider $w\in \mathcal{N}_{\lambda}^-$ fixed. According to \eqref{N^-2} we deduce that $\phi_{\lambda,w}''(1)\leq 0$. Moreover, using \eqref{fibra de J''}, we ensure that
\begin{equation}\label{d27}
\Vert w\Vert^2-(q-1)\lambda A(w)-(2p-1)B(w)\leq 0.
 \end{equation}
In light of \eqref{fibra de J'} we infer that
$
\lambda A(w)=\Vert u\Vert^2-B(w).
$
Under these conditions, by using \eqref{d27}, we see that
\begin{equation}\label{d29}
\Vert w\Vert^2\leq \displaystyle\frac{B(w)(2p-q)}{2-q}.
\end{equation}
Now, by using the Stein-Weiss inequality with $s=r=\frac{2N}{2N-2\alpha-\mu}$, we verify that $\frac{\alpha}{N}<1-\frac{1}{r}$ and
$$
B(w)\leq c\Vert b\vert u\vert^p\Vert_r^2.
$$
Recall also that $2< pr< 2^*$. In virtue of Hölder's inequality and Proposition \ref{imerso cont e comp} we deduce that
\begin{equation}\label{d30}
\Vert b\vert w\vert^p\Vert_r^r\leq \displaystyle\int_{\mathbb R^N}b^r\vert w\vert^{pr}dx\leq\Vert b\Vert_{\infty}^r\Vert w\Vert_{pr}^{pr}
\leq \Vert b\Vert_{\infty}^rS^{pr}_{pr}\Vert w\Vert^{pr}.
\end{equation}
It follows from \eqref{d29} and \eqref{d30} that
$$\Vert w\Vert\geq\left(\displaystyle\frac{1}{c_1}\right)^{\frac{1}{pr-2}}>0,\, \, w \in \mathcal{N}_{\lambda}^-$$
where $c_1=c_1(p,q,r) > 0.$ This completes the proof.
\end{proof}

\begin{lemma}\label{J(u)=C_ N-1}
 Suppose ($H_0$), ($V_1$), ($V_2$) and $\lambda\in (0,\lambda^*)$. Consider $(w_k)\subseteq\mathcal{N}_{\lambda}^-$ a minimizer sequence for the functional $ J_{\lambda}$ in $\mathcal{N}_{\lambda}^-.$ Then there exists $ w_\lambda \in \mathcal{N}_{\lambda}^- $such that, up to a subsequence, $ w_k \to w_\lambda$ in $X$. Furthermore, we obtain that $C_{\mathcal{N}_{\lambda}^-}=J_{\lambda}(w_\lambda)$.
\end{lemma}
\begin{proof}
   Firstly, we shall prove that there exists $w\in\mathcal{N}^-_{\lambda}$ such that $ J_{\lambda}(w)= C_{\mathcal{N}_{\lambda}^-}$. Let us consider $(w_k)\subseteq \mathcal{N}_{\lambda}^-$ a sequence such that $J_{\lambda}(w_k)\to C_{\mathcal{N}_{\lambda}^-}$. Since $J_{\lambda}$ is coercive in the Nehari set we obtain that $(w_k)$ is bounded. Hence, up to a subsequence, $w_k\rightharpoonup w$ in $X$.
  Now, we claim that $w\neq 0$. Otherwise, $w=0$ and
  $w_k\rightharpoonup 0$ in $X$. Consequently, $A(u_k)\to 0$ and $B(u_k)\to 0$. Under these conditions, by using \eqref{fibra de J'}, we have that $\Vert w_k\Vert\to 0$
which is a contradiction with Lemma \ref{N^- longe do zero}. Therefore, we deduce that $w\neq 0$.

 At this stage, we shall prove $w_k\to w$ in $X$. Once again the proof follows by arguing by contradiction. Let us assume that $w_k\rightharpoonup w$ in $X$ and $w_k\not\rightarrow w$ in $X$. As a consequence,  $$\Vert w\Vert<\displaystyle\liminf_{k\to \infty}\Vert w_k\Vert.$$
Since $w\neq 0$ there exists $t_\lambda^-(w)>0$ such that $t_\lambda^-(w)w \in \mathcal{N}_{\lambda}^-$. Now, due to the fact that the map $w\mapsto J_{\lambda}'(w)w$ is weakly lower semicontinuous, we infer also that
$$ 0=J_{\lambda}'(t_\lambda^-(w)w)w<\displaystyle\liminf_{k\to\infty} J_{\lambda}'(t_\lambda^-(w)w_k)w_k. $$
Therefore, $J_{\lambda}'(t_\lambda^-(w)w_k)w_k>0$ holds for each $k$ large enough. Recall also that the map $t \mapsto J_\lambda(t w_k)$ is increasing in the interval $ (t_\lambda^+(w_k),t_\lambda^-(w_k))$. In particular, by using the fact that $t_\lambda^-(w)\in (t_\lambda^+(w_k),t_\lambda^-(w_k))$ and taking into account that $w_k \in \mathcal{N}_{\lambda}^-$, we deduce that
$$C_{\mathcal{N}_{\lambda}^-}\leq J_{\lambda}(t_\lambda^-(w)w)<\displaystyle \lim_{k \to \infty}\inf J_{\lambda}(t_\lambda^-(w)w_k)\leq \displaystyle \lim_{k \to \infty}\inf J_{\lambda}(w_k)=C_{\mathcal{N}_{\lambda}^-}.$$
This is a contradiction proving that $w_k \to w$ in $X.$ This ends the proof.
\end{proof}

\begin{lemma}\label{L1}
Suppose ($H_0$), ($V_1$), ($V_2$) and $\lambda\in (0,\lambda^*)$. Let $u_{\lambda}\in\mathcal{N_{\lambda}^+}$ and $w_{\lambda}\in \mathcal{N}_{\lambda}^-$ such that $J_{\lambda} (u_{\lambda})=C_{\mathcal{N}_{\lambda}^+}$ and $J_{\lambda} (w_{\lambda})=C_{\mathcal{N}_{\lambda}^-}.$ Then, for each $\psi\in X_+$, there exists $\varepsilon_0>0$ such that
    \begin{itemize}
        \item [$a)$] $J_{\lambda}(u_{\lambda}+\varepsilon\psi)\geq J_{\lambda}(u_{\lambda}), 0\leq\varepsilon\leq\varepsilon_0$;
        \item[$b)$] $t_{\lambda}^-(w_{\lambda}+\varepsilon\psi)\to 1$ as $\varepsilon\to 0$ where $t^-_{\lambda}(w_{\lambda}+\varepsilon\psi)$ is the unique positive number such that $t^-_{\lambda}(w_{\lambda}+\varepsilon\psi)(w_{\lambda}+\varepsilon\psi) \in \mathcal{N}_\lambda^-$ holds for each $\psi\in X_+, 0\leq\varepsilon\leq\varepsilon_0$.
    \end{itemize}
    \begin{proof}

$a)$ Let us consider any fixed function $\psi\in X_+$. It follows from \eqref{fibra de J''} that
$$\phi''_{\lambda,u_{\lambda}+\varepsilon\psi}(1)=\Vert u_{\lambda}+\varepsilon\psi\Vert^2-(q-1)\lambda A(u_{\lambda}+\varepsilon\psi)-(2p-1)B(u_{\lambda}+\varepsilon\psi).$$
Notice also that the function $\varepsilon\mapsto\phi_{\lambda,u_{\lambda}+\varepsilon\psi}$ is continuous. On the other hand, we know that $\phi''_{\lambda,u_{\lambda}}(1)>0$ and $u_{\lambda}\in \mathcal{N}_{\lambda}^+$. Hence, there exists $\varepsilon_0>0$ such that $\phi''_{\lambda,u_{\lambda}+\varepsilon\psi}(1)> 0$ for any $0\leq\varepsilon\leq \varepsilon_0$. Under these conditions, taking into account Proposition \ref{analise dos lambidas} and Lemma \ref{J(u)=C_ N+1}, we deduce that
 \begin{eqnarray*}
J_{\lambda}( u_{\lambda}+\varepsilon\psi) \nonumber &=& \phi_{\lambda, u_{\lambda}+\varepsilon\psi}(1)\geq \phi_{\lambda, u_{\lambda}+\varepsilon\psi}(t_{\lambda}^+(u_{\lambda}+\varepsilon\psi)) = J_{\lambda}(t_{\lambda}^+(u_{\lambda}+\varepsilon\psi)(u_{\lambda}+\varepsilon\psi))\geq J_{\lambda}(u_{\lambda}).
\end{eqnarray*}
$b)$ Define the auxiliary function $F:(0,\infty)\times\mathbb R^3\to \mathbb R$ given by $F(t,e,f,g)= et-\lambda t^{q-1}f-t^{2p-1}g$.
Recall also that
$$F(1,e_1,f_1,g_1)=e_1-\lambda f_1-g_1=\phi'_{\lambda, w_{\lambda}}(1)=0.$$
Here we emphasize that $e_1=\Vert w_{\lambda}\Vert^2$, $f_1=A(w_{\lambda})$ and $g_1=B(w_{\lambda})$. Now, by using the fact that $w_{\lambda}\in \mathcal{N}_{\lambda}^-$, we mention that
$$\displaystyle\frac{dF}{dt}(1,e_1,f_1,g_1)=e_1-(q-1)\lambda f_1-(2p-1)g_1=\phi''_{\lambda, w_{\lambda}}(1)<0.$$
It follows from the Implicit Function Theorem (see \cite[Remark 4.2.3]{DRABEK}) that there exists $\epsilon_1>0$ and
$t_\epsilon=t(\Vert w_{\lambda}+\varepsilon\psi\Vert^2,A(w_{\lambda}+\varepsilon\psi),B(w_{\lambda}+\varepsilon\psi)),~0<\epsilon<\epsilon_1$ satisfying
$$F(t_\epsilon,\Vert w_{\lambda}+\varepsilon\psi\Vert^2, A(w_{\lambda}+\varepsilon\psi), B(w_{\lambda}+\varepsilon\psi))=0\qquad\mbox{and}\qquad\displaystyle\frac{dF}{dt}(t_\epsilon,\Vert w_{\lambda}+\varepsilon\psi\Vert^2,A(w_{\lambda}+\varepsilon\psi),B(w_{\lambda}+\varepsilon\psi))<0.$$
It follows also from Proposition \ref{analise dos lambidas} that $t_\epsilon=t^-_{\lambda}(w_{\lambda}+\varepsilon\psi)$ holds
for each $\varepsilon>0$ small enough. Moreover, we deduce that $t: \mathcal{B} \to \mathcal{A}$ is  in $C^{1}$ class where  $\mathcal{A}$ and $\mathcal{B}$ are open neighborhoods of $1$ and $(e_1, f_1, g_1)$, respectively. Furthermore, by using the continuity of $\varepsilon \to t_{\varepsilon}$ we deduce that
$t_{\lambda}^-(w_{\lambda}+\varepsilon\psi)\to 1=t^-_{\lambda}(w_{\lambda})$ as $\varepsilon\to 0.$
This finishes the proof.
\end{proof}
\end{lemma}
\begin{lemma}\label{L2}
 Suppose ($H_0$), ($V_1$), ($V_2$) and $\lambda\in (0,\lambda^*)$. Let $u_{\lambda}\in\mathcal{N_{\lambda}^+}$ and $w_{\lambda}\in \mathcal{N}_{\lambda}^-$ such that $J_{\lambda} (u_{\lambda})=C_{\mathcal{N}_{\lambda}^+}$ and $J_{\lambda} (w_{\lambda})=C_{\mathcal{N}_{\lambda}^-}.$ Then, for each $\varphi \in X_+$, we obtain that $a\vert u_{\lambda}\vert^{q-2}u\varphi$ and $a\vert w_{\lambda}\vert^{q-2}w\varphi \in L^1(\mathbb R^N).$ Moreover, we obtain that
\begin{equation}\label{E1}
\displaystyle\int\limits_{\mathbb R^N}(\nabla u_{\lambda}\nabla \varphi+V(x)u_{\lambda}\varphi) dx- \lambda H(u_{\lambda},\varphi)- D(u_{\lambda},\varphi)\geq 0,
\end{equation}
and
\begin{equation}\label{E2}
\displaystyle\int\limits_{\mathbb R^N}(\nabla w_{\lambda}\nabla \varphi+V(x)w_{\lambda}\varphi) dx- \lambda H(w_{\lambda},\varphi)- D(w_{\lambda},\varphi)\geq 0,
\end{equation}
where $D$ and $H$ were defined in \eqref{HD(u,phi)}. Furthermore, $u_{\lambda},w_{\lambda}>0$ a.e in $\mathbb R^N$.
\end{lemma}
\begin{proof}
    Let us consider a fixed function $\varphi\in X_+$. Firstly, we shall prove the inequality given in \eqref{E1}. In light of Lemma \ref{L1} item $a)$ we know that
    $
        J_{\lambda}(u_\lambda+\varepsilon\varphi)\geq J_{\lambda}(u_\lambda),\;\;\varepsilon\in (0,\varepsilon_0).
    $
    Consequently, we obtain that
    \begin{eqnarray*}
        \displaystyle\frac{1}{2}\Vert u_\lambda+\varepsilon\varphi\Vert^2-\displaystyle\frac{\lambda}{q}A(u_\lambda+\varepsilon\varphi)-\displaystyle\frac{1}{2p}B(u_\lambda+\varepsilon\varphi)\geq \displaystyle\frac{1}{2}\Vert u_\lambda\Vert^2-\displaystyle\frac{\lambda}{q}A(u_\lambda)-\displaystyle\frac{1}{2p}B(u_\lambda),
    \end{eqnarray*}
where $A$ and $B$ was defined in \eqref{AB(u)2}. As a consequence, we infer that
    \begin{eqnarray}\label{deriv-J}
        \displaystyle\frac{\lambda} {q}\displaystyle\frac{A(u_\lambda+\varepsilon\varphi)-A(u_\lambda)}{\varepsilon}\leq \displaystyle\frac{\Vert u_\lambda+\varepsilon\varphi\Vert^2-\Vert u_\lambda\Vert^2}{2\varepsilon}-\displaystyle\frac{1}{2p}\displaystyle\frac{B(u_\lambda+\varepsilon\varphi)-B(u_\lambda)}{\varepsilon}
    \end{eqnarray}
  holds for $\varepsilon>0$ small enough. Recall also that $\|.\|$ and $B$ are in $C^1$ class. Hence, the right-hand side given in \eqref{deriv-J} is uniformly bounded from above by a positive constant $M_\varphi$ which does not depend on $\epsilon$. Now, using the Mean Value Theorem, there exists $0<\theta=\theta(\epsilon)< \epsilon<1$ such that
    $$\int_{\R^N} a(x)(u_\lambda+\theta\varphi)^{q-1}\varphi<\frac{M_\varphi}{\lambda}.$$
It follows from Fatou's Lemma that
    $$\int_{\R^N} a(x)u_\lambda^{q-1}\varphi<\frac{M_\varphi}{\lambda}.$$

Now, we proceed to argue by contradiction. Assume that $|[u_\lambda=0]|>0$ holds. Given any $\delta>0$ we infer that
\begin{equation}\label{maxwell}
\frac{M_\varphi}{\lambda}>\displaystyle\int_{\mathbb R^N}a(x) u_\lambda^{q-1}\varphi dx>\displaystyle\int_{[u_\lambda<\delta]}a(x) u_\lambda^{q-1}\varphi dx>\displaystyle\frac{1}{\delta^{1-q}}\displaystyle\int_{[u_\lambda=0]}a(x)\varphi\to \infty,
\end{equation}
as $\delta\to 0$. This is a contradiction proving that $|[u_\lambda=0]|=0$ holds.
Thus $u_\lambda>0$ a.e in $\mathbb R^N$ and $a(x) u_\lambda^{q-1}\varphi\in L^1(\mathbb R^N)$ is satisfied for each $\varphi \in X_+$. Furthermore, by using \eqref{deriv-J}, we ensure that \eqref{E1} is now verified.

    Now, we are going to prove \eqref{E2}. It is important to observe that
    \begin{eqnarray*}
         J_{\lambda}(t_{\lambda}^-(w_\lambda+\varepsilon\varphi)(w_\lambda+\varepsilon\varphi))\geq J_{\lambda}(w_\lambda)=\phi_{\lambda,w_\lambda}(1)
        \geq \phi_{\lambda,w_\lambda}(t_{\lambda}^-(w_\lambda+\varepsilon\varphi))=J_{\lambda}(t_{\lambda}^-(w_\lambda+\varepsilon\varphi)w_\lambda).
    \end{eqnarray*}
   In particular, we infer that
   \begin{eqnarray*}
        \displaystyle\frac{1}{2}t^-_{\lambda}(w_\lambda+\varepsilon\varphi)^2\Vert w_\lambda+\varepsilon\varphi\Vert^2&-&\displaystyle\frac{\lambda}{q}t^-_{\lambda}(w_\lambda+\varepsilon\varphi)^q A(w_\lambda+\varepsilon\varphi)-\displaystyle\frac{1}{2p}t^-_{\lambda}(w_\lambda+\varepsilon\varphi)^{2p} B(w_\lambda+\varepsilon\varphi)\\
        &\geq& \displaystyle\frac{1}{2}t^-_{\lambda}(w_\lambda+\varepsilon\varphi)^2\Vert w_\lambda\Vert^2-\displaystyle\frac{\lambda}{q}t^-_{\lambda}(w_\lambda+\varepsilon\varphi)^q A(w_\lambda) - \displaystyle\frac{1}{2p}t^-_{\lambda}(w_\lambda+\varepsilon\varphi)^{2p} B(w_\lambda).
    \end{eqnarray*}
   The last assertion implies that
    \begin{equation}\label{E3}
        t^-_{\lambda}(w_\lambda+\varepsilon\varphi)^2\left[\frac{\Vert w_\lambda+\varepsilon\varphi\Vert^2-\Vert w_\lambda\Vert^2}{2\varepsilon}\right]-t^-_{\lambda}(w_\lambda+\varepsilon\varphi)^{2p}\left[\frac{B(w_\lambda+\varepsilon\varphi)-B(w_\lambda)}{2p\varepsilon}\right]
       \geq t^-_{\lambda}(w_\lambda+\varepsilon\varphi)^q \lambda\frac{A(w_\lambda+\varepsilon\varphi)-A(w_\lambda)}{q\varepsilon},\\
    \end{equation}
    where $\varepsilon>0$ is small enough. Under these conditions, by using \eqref{maxwell} and Lemma \ref{L1} (b), we consider the limit $\varepsilon\to 0$ proving that $w_\lambda>0$ a.e. in $\mathbb R^N$ and $a(x)w_\lambda^{q-1}\varphi\in L^1(\mathbb R^N)$ holds for each $\varphi \in X_+$. Therefore, by using \eqref{E3} and doing $\varepsilon\to 0$, we obtain
    $$\left<w_\lambda,\varphi\right>-\lambda H(w_\lambda,\varphi)-D(w_\lambda,\varphi)\geq 0,\;\;\varphi \in X_+.$$
    This completes the proof.
    \end{proof}
\begin{prop}\label{solucao P1}
     Suppose ($H_0$), ($V_1$), ($V_2$) and $\lambda\in (0,\lambda^*)$. Let $u_{\lambda}\in\mathcal{N_{\lambda}^+}$ and $w_{\lambda}\in \mathcal{N}_{\lambda}^-$ such that $J_{\lambda} (u_{\lambda})=C_{\mathcal{N}_{\lambda}^+}$ and $J_{\lambda} (w_{\lambda})=C_{\mathcal{N}_{\lambda}^-}.$ Then $u_{\lambda}$ and $w_{\lambda}$ are weak solutions to the Problem \eqref{P1}.
\end{prop}
    \begin{proof}
        First, we show that $u_{\lambda}$ is a weak solution to the Problem \eqref{P1}.
        Given any fixed function $\varphi\in X$ define $\Psi=(u_{\lambda}+\varepsilon\varphi)^+\in X_+$ where $\varepsilon>0$ is small enough. Notice also that $\Psi\in X_+$. Now, by using Lemma \ref{L2} and \eqref{E1}, we deduce that
        \begin{eqnarray}\label{E4}
        0 \nonumber &\leq &\displaystyle\int_{\mathbb R^N}(\nabla u_{\lambda}\nabla \Psi+V(x)u_{\lambda}\Psi) dx- \lambda\displaystyle\int_{\mathbb R^N}a(x)\vert u_{\lambda}^{q-2}u_{\lambda}\Psi dx- \displaystyle\int_{\mathbb R^N}\displaystyle\int_{\mathbb R^N}\displaystyle\frac{b(y)\vert u_{\lambda}\vert^pb(x)\vert u_{\lambda}\vert^{p-2}u_{\lambda}\Psi}{\vert x\vert^{\alpha}\vert x-y\vert^{\mu}\vert y\vert^{\alpha}}dxdy\\
        \nonumber &=& \left<u_{\lambda},u_{\lambda}+\varepsilon\varphi\right>-\lambda\displaystyle\int_{\mathbb R^N} a(x)\vert u\vert^{q-2}u_{\lambda}(u_{\lambda}+\varepsilon\varphi)dx- \displaystyle\int_{\mathbb R^N}\displaystyle\int_{\mathbb R^N}\displaystyle\frac{b(y)\vert u_{\lambda}\vert^pb(x)\vert u_{\lambda}\vert^{p-2}u_{\lambda}(u_{\lambda}+\varepsilon\varphi)}{\vert x\vert^{\alpha}\vert x-y\vert^{\mu}\vert y\vert^{\alpha}}dxdy\\
   \nonumber &-& \left<u_{\lambda},(u_{\lambda}+\varepsilon\varphi)^-\right>+\lambda\displaystyle\int_{[u_{\lambda}+\varepsilon\varphi\leq 0]}a(x)\vert u_{\lambda}\vert^{q-2}u_{\lambda}(u_{\lambda}+\varepsilon\varphi)dx\\
   \nonumber &+& \displaystyle\int_{[u_{\lambda}+\varepsilon\varphi\leq 0]}\displaystyle\int_{\mathbb R^N}\displaystyle\frac{b(y)\vert u_{\lambda}\vert^pb(x)\vert u_{\lambda}\vert^{p-2}u_{\lambda}(u_{\lambda}+\varepsilon\varphi)}{\vert x\vert^{\alpha}\vert x-y\vert^{\mu}\vert y\vert^{\alpha}}dxdy\\
   \nonumber&\leq &  \Vert u_{\lambda}\Vert^2-\lambda A(u_{\lambda})-B(u_{\lambda})+\varepsilon\left<u_{\lambda}, \varphi\right>-\lambda\varepsilon\displaystyle\int_{\mathbb R^N}a(x)\vert u_{\lambda}\vert^{q-2}u\varphi dx- \varepsilon\displaystyle\int_{[u+\varepsilon\varphi\leq 0]}\nabla u_{\lambda}\nabla\varphi+V(x)u_{\lambda}\varphi dx\\
   &-& \varepsilon\displaystyle\int_{\mathbb R^N}\displaystyle\int_{\mathbb R^N}\displaystyle\frac{b(y)\vert u_{\lambda}\vert^p b(x)\vert u_{\lambda}\vert^{p-2}u_{\lambda}\varphi}{\vert x\vert^{\alpha}\vert x-y\vert^{\mu}\vert y\vert^{\alpha}}dxdy.
 \end{eqnarray}
 Recall also that $u_{\lambda}\in \mathcal{N_{\lambda}^+}$. It follows from \eqref{fibra de J'} that
 $$\phi_{\lambda,u_\lambda}'(1)=\Vert u_{\lambda}\Vert^2-\lambda A(u_{\lambda})-B(u_{\lambda})=0.$$
Thus, by using \eqref{E4}, we infer that
  \begin{eqnarray*}\label{E6}
     0\nonumber& \leq & \left<u_{\lambda}, \varphi\right>-\lambda\displaystyle\int_{\mathbb R^N}a(x)\vert u_{\lambda}\vert^{q-2}u\varphi dx- \displaystyle\int_{[u+\varepsilon\varphi\leq 0]\!}\nabla u_{\lambda}\nabla\varphi+V(x)u_{\lambda}\varphi dx- \displaystyle\int_{\mathbb R^N}\!\displaystyle\int_{\mathbb R^N}\!\displaystyle\frac{b(y)\vert u_{\lambda}\vert^p b(x)\vert u_{\lambda}\vert^{p-2}u_{\lambda}\varphi}{\vert x\vert^{\alpha}\vert x-y\vert^{\mu}\vert y\vert^{\alpha}}dxdy.
 \end{eqnarray*}
Now, taking into account \eqref{TCD-grad} and \eqref{d22}, we see that
$$ \lim_{\varepsilon\to 0}\displaystyle\int_{[u+\varepsilon\varphi\leq 0]}\nabla u_{\lambda}\nabla\varphi+V(x)u_{\lambda}\varphi dx=0.$$
Therefore, as $\varepsilon\to 0,$ we have that
 \begin{eqnarray}\label{E6}
     0\nonumber& \leq & \left<u_{\lambda}, \varphi\right>-\lambda\displaystyle\int_{\mathbb R^N}a(x)\vert u_{\lambda}\vert^{q-2}u\varphi dx- \displaystyle\int_{\mathbb R^N}\displaystyle\int_{\mathbb R^N}\displaystyle\frac{b(y)\vert u_{\lambda}\vert^p b(x)\vert u_{\lambda}\vert^{p-2}u_{\lambda}\varphi}{\vert x\vert^{\alpha}\vert x-y\vert^{\mu}\vert y\vert^{\alpha}}dxdy,\,\, \varphi\in X.
 \end{eqnarray}
As a consequence, $u_{\lambda}\in\mathcal{N}^+_{\lambda}$ is a weak solution for the Problem \eqref{P1}.
Similarly, we show that $w_{\lambda}\in \mathcal{N}^-_{\lambda}$ is a weak solution for the Problem \eqref{P1}. This ends the proof.
\end{proof}

\section{Proof of Theorem \ref{teorema 1}}\label{sec-teo1}
\begin{proof}
It follows from Proposition \ref{solucao P1} that Problem \eqref{P1} admits at least two distinct positive solutions which are denoted by
$u_{\lambda},w_{\lambda}\in X$ whenever $\lambda\in(0,\lambda^*)$. Now, by using Lemma \ref{J(u)=C_ N+1} and Lemma \ref{J(u)=C_ N-1}, we know that $u_{\lambda}\in \mathcal{N}_{\lambda}^+$ and $w_{\lambda}\in \mathcal{N}_{\lambda}^-$. In virtue of   Lemma \ref{N^- longe do zero} there exists a constant $c>0$ such that $\Vert w_{\lambda}\Vert\geq c$. It follows from Lemma \ref{C_n+<01} and Lemma \ref{J(u)=C_ N+1} that $J(u_{\lambda})=C_{\mathcal{N}_{\lambda}^+}<0$. On the other hand, by using Lemma \ref{t cont. em relação a lambda} item $b),$ the maps $\lambda\mapsto J_{\lambda}(u_{\lambda})$ and $\lambda\mapsto J_{\lambda}(w_{\lambda})$ are decreasing for each $\lambda\in(0, \lambda^*)$. The proof of the item $e)$ is analogous to the proof given in \cite[Theorem 1.1]{RAY2021}. This ends the proof.
\end{proof}

\section{Multiplicity of solutions for $\lambda=\lambda^*$}
In this section, we shall prove the existence of at least two positive solutions to the Problem \eqref{P1} whenever $\lambda=\lambda^*$. First, we show that the Problem ($Q_{\lambda^*}$) does not admit any weak solution $u\in\mathcal{N}^0_{\lambda^*}$. More specifically, we prove the following result:
\begin{prop}\label{sem sol}
 Suppose ($H_0$), ($H'_1$), ($V_1$) and ($V_2$).
Then the problem Problem \eqref{P1} does not admit any weak solution $u\in \mathcal{N}^0_{\lambda}$ where $\lambda=\lambda^*$.
\end{prop}
\begin{proof}
The proof follows arguing by contradiction. Let us assume that there is a weak solution $u\in \mathcal{N}^0_{\lambda^*}$ for the Problem \eqref{P1}. It follows from Lemma \ref{3 passos} that
    \begin{equation*}
        0 = 2\left< u,\varphi \right>-q\lambda^*\displaystyle\int_{\mathbb R^N}a(x)\vert u\vert^{q-2}u\varphi dx-2p D(u,\varphi),\;\;\varphi\in X.
    \end{equation*}
    It is important to emphasize that
    $$    D(u,\varphi)=\displaystyle\int_{\mathbb R^N}\displaystyle\int_{\mathbb R^N}\displaystyle\frac{b(y)\vert u\vert^p b(x)\vert u\vert^{p-2}u\varphi}{\vert x\vert^{\alpha}\vert x-y\vert^{\mu}\vert y\vert^{\alpha}}dxdy,\;\;u,\varphi\in X.
    $$
    Recall also that $u\in \mathcal{N}^0_{\lambda^*}$ is a weak solution to the Problem \eqref{P1}. Hence, \eqref{eq.sol.fraca} is satisfied for each $\varphi \in X$. In particular, we see that
    \begin{eqnarray*}
\left\{\begin{array}{ll}
    2\left< u,\varphi \right>-q\lambda^*\displaystyle\int_{\mathbb R^N}a(x)\vert u\vert^{q-2}u\varphi dx-2p D(u,\varphi)=0,\\
   -2\left< u,\varphi \right>+2\lambda^*\displaystyle\int_{\mathbb R^N}a(x)\vert u\vert^{q-2}u\varphi dx+2D(u,\varphi)=0,\;\;\varphi\in X.
\end{array}
\right.
\end{eqnarray*}
As a consequence, we obtain that
\begin{eqnarray*}
\displaystyle\int_{\mathbb R^N}\left[(2-q)\lambda^*a(x)\vert u\vert^{q-2}u+2(1-p)\displaystyle\int_{\mathbb R^N}\displaystyle\frac{b(y)\vert u\vert^p}{\vert x\vert^{\alpha}\vert x-y\vert^{\mu}\vert y\vert^{\alpha}}dyb(x)\vert u\vert^{p-2}u\right]\varphi dx=0,~\varphi\in X.
\end{eqnarray*}
Under these conditions, we see that
\begin{eqnarray*}
    (2-q)\lambda^*a(x)\vert u\vert^{q-2}u+2(1-p)\displaystyle\int_{\mathbb R^N}\displaystyle\frac{b(y)\vert u\vert^p}{\vert x\vert^{\alpha}\vert x-y\vert^{\mu}\vert y\vert^{\alpha}}dyb(x)\vert u\vert^{p-2}u=0\;\hbox{a.e. in }\;\mathbb R^N.
\end{eqnarray*}
The last implies that
\begin{equation}\label{d23}
a(x)=\displaystyle\frac{2(p-1)}{(2-q)\lambda^*}\displaystyle\int_{\mathbb R^N}\displaystyle\frac{b(y)\vert u\vert^p}{\vert a\vert^{\alpha}\vert x-y\vert^{\mu}\vert y\vert^{\alpha}}dy\;b(x)\vert u\vert^{p-q},\;x\in \mathbb R^N.
\end{equation}
Now, by using \eqref{d23} and integrating both sides, taking into account that $a\notin L^1(\mathbb R^N)$ we deduce that
\begin{equation}\label{d24}
   \infty= \displaystyle\int_{\mathbb R^N}a(x)dx=\displaystyle\frac{2(p-1)}{(2-q)\lambda^*}\displaystyle\int_{\mathbb R^N}\displaystyle\int_{\mathbb R^N}\displaystyle\frac{b(y)\vert u\vert^pb(x)\vert u\vert^{p-q}}{\vert x\vert^{\alpha}\vert x-y\vert^{\mu}\vert y\vert^{\alpha}}dydx.
\end{equation}
Now, we analyze the right-hand side given in \eqref{d24}. Using the Stein-Weiss inequality given in Proposition \ref{Stein-Weiss} with $r=s= 2N/(2N-2\alpha-\mu)$, we have that
\begin{equation}\label{d25}
\displaystyle\int_{\mathbb R^N}\displaystyle\int_{\mathbb R^N}\displaystyle\frac{b(y)\vert u\vert^pb(x)\vert u\vert^{p-q}}{\vert x\vert^{\alpha}\vert x-y\vert^{\mu}\vert y\vert^{\alpha}}dydx
\leq c_1 \Vert b\vert u\vert^p\Vert_r\Vert b\vert u\vert^{p-q}\Vert_s.
\end{equation}
Now, by using \eqref{d25}, applying Hölder's inequality and Proposition \ref{imerso cont e comp}, we deduce that
\begin{eqnarray*}
 \Vert b\vert u\vert^p\Vert_r\Vert b\vert u\vert^{p-q}\Vert_r\nonumber& = & \displaystyle\int_{\mathbb R^N}\vert b(x)\vert^r\vert u\vert^{pr}dx\displaystyle\int_{\mathbb R^N}\vert b(x)\vert^r\vert u\vert^{r(p-q)}dx\\
\nonumber&\leq & \Vert b\Vert_{\infty}^r\Vert u\Vert_{pr}^{pr}
\left(\displaystyle\int_{\mathbb R^N}\vert b(x)\vert^{rt_1}\right)^{\frac{1}{t_1}}\left(\displaystyle\int_{\mathbb R^N}\vert u\vert^{2^*}\right)^{\frac{r(p-q)}{2^*}}\\
\nonumber & \leq & S_{pr}^{pr}\Vert b \Vert_{\infty}^r
\Vert u\Vert^{pr}\Vert b\Vert_{rt_1}^r\Vert u\Vert^{r(p-q)}_{2^*} \leq  S_{pr}^{pr}S_{2^*}^{r(p-q)}\Vert b\Vert^r_{\infty}\Vert u\Vert^{pr}\Vert b\Vert_{rt_1}^r\Vert u\Vert^{r(p-q)},
\end{eqnarray*}
where $t_1$ is the conjugate exponent of $\frac{2^*}{r(p-q)}$. More precisely, we obtain
$$t_1=\displaystyle\frac{2N-2\alpha-\mu}{2N-2\alpha-\mu-(N-2)(p-q)},\qquad rt_1=\displaystyle\frac{2N}{2N-2\alpha-\mu-(N-2)(p-q)}\qquad\mbox{and}\qquad 2\leq pr\leq 2^*.$$
Now, we consider $C= S_{pr}^{pr}S_{2^*}^{r(p-q)}\Vert b\Vert^r_{\infty}$ which is finite, see hypothesis ($H_1'$). In view of the last assertion we obtain that
\begin{eqnarray}\label{d40}
\displaystyle\int_{\mathbb R^N}\displaystyle\int_{\mathbb R^N}\displaystyle\frac{b(y)\vert u\vert^pb(x)\vert u\vert^{p-q}}{\vert x\vert^{\alpha}\vert x-y\vert^{\mu}\vert y\vert^{\alpha}}dydx
\leq \Vert b\vert u\vert^p\Vert_r\Vert b\vert u\vert^{p-q}\Vert_r
\leq  C\Vert b\Vert^r_{\infty}\Vert b\Vert_{rt_1}^r\Vert u\Vert^{r(2p-q)}<\infty.
\end{eqnarray}
However, we obtain a contradiction by using \eqref{d40} and \eqref{d24}. This ends the proof.
\end{proof}

 It follows from Theorem \ref{teorema 1} that there exists weak solutions for the Problem \eqref{P1} which are denoted by $u_{\lambda}\in \mathcal{N}_{\lambda}^+$ and $w_{\lambda}\in\mathcal{N}_{\lambda}^-$ for each $\lambda\in (0,\lambda^*)$. The next result is fundamental in order to prove that Problem \eqref{P1} admit at least two weak solutions $u_{\lambda^*}\in \mathcal{N}_{\lambda^*}^+$ and $w_{\lambda^*}\in\mathcal{N}_{\lambda^*}^-$. To do that we apply some ideas discussed in \cite[Theorem 1.2]{RAY2021}. In fact, we shall prove the following result:

\begin{prop}\label{C_n decrescente}
 Suppose ($H_0$), ($H'_1$), ($H'_2$), ($V_1$) and ($V_2$). Assume that $\Tilde{\lambda}\in (0,\lambda^*)$ and $(\lambda_k)\subset (0,\lambda^*)$ such that $\lambda_k\to\Tilde{\lambda}$. Then we obtain the following statements:
\begin{itemize}
    \item [$a)$] The functions $\lambda\mapsto u_{\lambda}$ and $\lambda\mapsto w_{\lambda}$ are continuous, i.e., for each sequence $\lambda_k\to \tilde{\lambda}$ we obtain that $u_{\lambda_k}\to u_{\tilde{\lambda}}$ and $w_{\lambda_k}\to w_{\tilde{\lambda}}$ in $X$ as $k\to +\infty$ where $J_{\tilde{\lambda}}(u_{\tilde{\lambda}})=C_{\mathcal{N}_{\tilde{\lambda}}^+}$ and $J_{\tilde{\lambda}}(w_{\tilde{\lambda}})=C_{\mathcal{N}_{\tilde{\lambda}}^-}.$
\item[$b)$] The functions $(0,\lambda^*]\ni \lambda\to C_{\mathcal{N}_{\lambda}^{\pm}}$ are decreasing and left-continuous for each $\lambda\in (0,\lambda^*)$;
\item [$c)$] There holds $\displaystyle\lim_{\lambda \uparrow \lambda^*} C_{\mathcal{N}_{\lambda}^{\pm}}=C_{\mathcal{N}_{\lambda^*}^{\pm}}$.
\end{itemize}
\begin{proof}
Consider two sequences $(u_{\lambda_k})$ and $(w_{\lambda_k})$ which are weak solutions for the Problem $(Q_{\lambda_k})$ where $k\in \mathbb N$ and $\lambda_k\in (0,\lambda^*)$. Notice also that $J_{\lambda_k}(u_{\lambda_k})=C_{\mathcal{N}_{\lambda_k}^+}$ and $J_{\lambda_k}(w_{\lambda_k})=C_{\mathcal{N}_{\lambda_k}^-}$ where $(u_{\lambda_k})$ and $(v_{\lambda_k})$ are minimizers on the Nehari sets $\mathcal{N}_{\lambda_k}^+$ and $\mathcal{N}_{\lambda_k}^-,$ respectively. Let us prove that $(u_{\lambda_k})$ and $(w_{\lambda_k})$ are bounded. The proof follows arguing by contradiction. Assume that $\Vert u_{\lambda_k}\Vert\to \infty$ as $k\to +\infty$. Using the identity $J_{\lambda_k}'(u_{\lambda_k})u_{\lambda_k}=0$ we know that
\begin{equation}\label{D5}
B(u_{\lambda_k})=\Vert u_{\lambda_k}\Vert^2+\lambda_kA(u_{\lambda_k}).
\end{equation}
Now, applying Hölder's inequality, Lemma \ref{imerso cont e comp} and \eqref{D5}, we obtain the following estimates:
\begin{eqnarray}\label{resultado}
J_{\lambda_k}(t_{\lambda_k}^+(u_{\lambda_k})u_{\lambda_k})=J_{\lambda_k}(u_{\lambda_k})
&=& \left(\displaystyle\frac{1}{2}-\displaystyle\frac{1}{2p}\right)\Vert u_{\lambda_k}\Vert^2-\lambda_k\left(\displaystyle\frac{1}{q}-\displaystyle\frac{1}{2p}\right)A(u_{\lambda_k})\nonumber\\
&\geq& \left(\displaystyle\frac{1}{2}-\displaystyle\frac{1}{2p}\right)\Vert u_{\lambda_k}\Vert^2-\lambda_k\left(\displaystyle\frac{1}{q}-\displaystyle\frac{1}{2p}\right)S_2^q\Vert a\Vert_r\Vert u_{\lambda_k}\Vert^q.
\end{eqnarray}
On the other hand, by using Lemma \ref{t cont. em relação a lambda}, we infer that
\begin{equation}\label{4.1}
\displaystyle\limsup_{k\to\infty} C_{\mathcal{N}_{\lambda_k}^+}\leq \displaystyle\limsup_{k\to \infty} J_{\lambda_k}(t_{\lambda_k}^+(u_{\tilde{\lambda}})u_{\tilde{\lambda}})=J_{\tilde{\lambda}}(t_{\tilde{\lambda}}^+(u_{\tilde{\lambda}})u_{\tilde{\lambda}})=J_{\tilde{\lambda}}(u_{\tilde{\lambda}})=C_{\mathcal{N}_{\tilde{\lambda}}^+}.
\end{equation}
Thus, using \eqref{4.1} and \eqref{resultado}, we see that
$$
C_{\mathcal{N_{\tilde{\lambda}}^+}}\geq\displaystyle\limsup_{k\to\infty} J_{\lambda_k}(u_{\lambda_k})\geq \displaystyle\limsup_{k\to \infty}\left[\left(\displaystyle\frac{1}{2}-\displaystyle\frac{1}{2p}\right)\Vert u_{\lambda_k}\Vert^2-\lambda_k\left(\displaystyle\frac{1}{q}-\displaystyle\frac{1}{2p}\right)C\Vert u_{\lambda_k}\Vert^q\right].
$$
Since $0<q<1$ and $C>0$ the last assertion implies that $(u_{\lambda_k})$ is bounded. Analogously, we prove that $(w_{\lambda_k})$ is also bounded. This ends the proof of the claim.

Now, up to a subsequence, there exists $u_{\tilde{\lambda}}\in X$ such that $u_{\lambda_k}\rightharpoonup u_{\tilde{\lambda}}$ in $X$.
Define the function $J:(0, +\infty) \times X \to \mathbb{R}$ given by $J(\lambda, u) := J_{\lambda} (u)$. Now, by using Proposition \ref{J continuo1}, we show that $u \mapsto J_{\lambda}(u)$ is continuous. It is not hard to see that $\lambda \mapsto J_\lambda(u)$ is also continuous. Let $\lambda_k\to\tilde{\lambda} \in (0,\lambda^*]$ be a fixed sequence. Recall that $(u_{\lambda_k})$ is a sequence of solutions for the Problem $(Q_{\lambda_k})$, i.e.,
 \begin{equation}\label{1.1}
     \left\langle u_{\lambda_k},\varphi\right\rangle =D(u_{\lambda_k},\varphi) +\lambda_k H(u_{\lambda_k},\varphi),\;\;\varphi\in X.
 \end{equation}
Under these conditions, we shall prove that $\Vert u_{\lambda_k}-u_{\tilde{\lambda}}\Vert\to 0$. It is important to mention that $u_{\lambda_k}\rightharpoonup u_{\lambda^*}$ in $X$. Hence,  we obtain that
\begin{eqnarray*}\label{E11}
   \displaystyle\limsup_{k\to \infty}\Vert u_{\lambda_k}-u_{\tilde{\lambda}}\Vert^2 &\leq & \displaystyle\limsup_{k\to\infty} \left<u_{\lambda_k},u_{\lambda_k}-u_{\tilde{\lambda}}\right>+\displaystyle\limsup_{k\to\infty} \left<-u_{\tilde{\lambda}},u_{\lambda_k}-u_{\tilde{\lambda}}\right>. \nonumber\\
\end{eqnarray*}
Now, by applying \eqref{1.1} and using the test function $u_{\lambda_k}-u_{\tilde{\lambda}}$, we infer that
\begin{equation*}
  \left<u_{\lambda_k},u_{\lambda_k}-u_{\tilde{\lambda}}\right>=D(u_{\lambda_k},u_{\lambda_k}-u_{\tilde{\lambda}}) +\lambda_k H(u_{\lambda_k},u_{\lambda_k}-u_{\tilde{\lambda}}).
\end{equation*}
As a consequence, we mention that
\begin{eqnarray}\label{F2}
  \displaystyle\limsup_{k\to \infty}\left<u_{\lambda_k},u_{\lambda_k}-u_{\tilde{\lambda}}\right>\nonumber &\leq & \displaystyle\limsup_{k\to \infty} D(u_{\lambda_k},u_{\lambda_k}-u_{\tilde{\lambda}})+ \displaystyle\limsup_{k\to \infty} \lambda_k H(u_{\lambda_k},u_{\lambda_k}-u_{\tilde{\lambda}}).
\end{eqnarray}
On the other hand, using the Stein-Weiss inequality with $s=r=2N/(2N-2\alpha-\mu)$ and taking into account the Proposition \ref{DHLSP}, we deduce that
\begin{eqnarray}\label{F1}
  D(u_{\lambda_k},u_{\lambda_k}-u_{\tilde{\lambda}}) \nonumber &=&\displaystyle\int_{\mathbb R^n}\displaystyle\int_{\mathbb R^n}\displaystyle\frac{b(y) u_{\lambda_k}^p b(x) u_{\lambda_k}^{p-1}(u_{\lambda_k}-u_{\tilde{\lambda}})}{\vert x\vert^{\alpha}\vert x-y\vert^{\mu}\vert y\vert^{\alpha}}dxdy\\
  \nonumber & \leq & C\Vert b\Vert_{\infty}^r\Vert u_{\lambda_k}\Vert_{pr}^p\Vert b u_{\lambda_k}^{p-1}(u_{\lambda_k}-u_{\tilde{\lambda}})\Vert_{r} \leq  C\Vert b\Vert_{\infty}^rS_{pr}^p\Vert u_{\lambda_k}\Vert^p\Vert b u_{\lambda_k}^{p-1}(u_{\lambda_k}-u_{\lambda^*})\Vert_{r}.
\end{eqnarray}
It is not hard to verify that $2< pr< 2^*$ and $\frac{\alpha}{N}<1-\frac{1}{r}$ hold. Using Hölder's inequality and Proposition \ref{imerso cont e comp} we infer that
\begin{eqnarray}
 \Vert b u_{\lambda_k}^{p-1}(u_{\lambda_k}-u_{\tilde{\lambda}})\Vert_r &\leq &
 \Vert b\Vert_{rt_1}^r\Vert u_{\lambda_k}\Vert_{2^*}^{r(p-1)}\Vert u_{\lambda_k}-u_{\tilde{\lambda}} \Vert_{rt_2}^r
 \leq  \Vert b\Vert_{rt_1}^rS_{2^*}^{r(p-1)}\Vert u_{\lambda_k}\Vert^{r(p-1)}\Vert u_{\lambda_k}-u_{\tilde{\lambda}} \Vert_{rt_2}^r\to 0,
 \end{eqnarray}
where $t_2=2^*/r$ and $t_1$ satisfies $$\displaystyle\frac{1}{t_1}+\displaystyle\frac{r(p-1)}{2^*}+\displaystyle\frac{1}{t_2}=1.$$
Consequently, we obtain the following identities
$$t_1=\displaystyle\frac{2^*}{2^*-rp}\qquad\mbox{and}\qquad rt_1=\displaystyle\frac{2N}{2N-2\alpha-\mu-(N-2)p}.$$
According to ($H_2'$) we obtain that $b\in L^{rt_1}(\mathbb R^N).$ It follows from \eqref{F1} that
$$\displaystyle\limsup_{k\to \infty} D(u_{\lambda_k},u_{\lambda_k}-u_{\tilde{\lambda}})=0.$$

From now on, we observe that there exists a measurable function $h_2$ such that $\left\vert a\vert u_{\lambda_k}\right\vert^q\vert\leq a h_2^q \in L^1(\mathbb R^N)$. Hence, using the Dominated Convergence Theorem, Lemma \ref{Jsemi contínua} and Fatou's Lemma, we prove that
\begin{eqnarray*}
    \displaystyle\limsup_{k\to \infty} \lambda_k H(u_{\lambda_k},u_{\lambda_k}-u_{\tilde{\lambda}}) \nonumber &=&\displaystyle\limsup_{k\to \infty}\lambda_k\displaystyle\int_{\mathbb R^N}a(x) u_{\lambda_k}\vert^{q-1}(u_{\lambda_k}-u_{\tilde{\lambda}})\;dx\nonumber\\
    \nonumber &\leq & \displaystyle\limsup_{k\to \infty} \lambda_k\displaystyle\int_{\mathbb R^N}a(x) u_{\lambda_k}^q\;dx
    - \displaystyle\liminf_{k\to \infty} \lambda_k\displaystyle\int_{\mathbb R^N}a(x) u_{\lambda_k}^{q-1}u_{\tilde{\lambda}}\;dx\\
\end{eqnarray*}
As a consequence, by using the last assertion and \eqref{F2}, we infer that
$$\displaystyle\limsup_{k\to \infty}\left<u_{\lambda_k},u_{\lambda_k}-u_{\tilde{\lambda}}\right>\leq 0.$$
Under these conditions, we deduce that
$$\displaystyle\liminf_{k\to \infty}\Vert u_{\lambda_k}-u_{\tilde{\lambda}}\Vert^2\leq\limsup_{k\to \infty}\Vert u_{\lambda_k}-u_{\tilde{\lambda}}\Vert^2=0.$$
Thus, we obtain that $u_{\lambda_k}\to u_{\tilde{\lambda}}$ in $X$. The same argument can be applied in order to guarantee that $(w_{\lambda_k})$ strongly converges to some $w \in X$.

Now, we shall show that $u_{\tilde{\lambda}}\in \mathcal{N}_{\tilde{\lambda}}^+$ and $w_{\tilde{\lambda}}\in \mathcal{N}_{\tilde{\lambda}}^-.$  In virtue of \eqref{4.1} and the continuity of $J$ we observe that
$$\displaystyle\limsup_{k\to\infty} C_{\mathcal{N}_{\lambda_k}^+}= J_{\tilde{\lambda}}(u_{\tilde{\lambda}})= C_{\mathcal{N}_{\tilde{\lambda}}^+}<0.$$
Therefore, we obtain that $u_{\tilde{\lambda}}\neq 0$. Furthermore, we mention that $u_{\lambda_k}\in \mathcal{N}_{\lambda_k}^+$.
Using the last assertion we infer that
$$J'(u_{\tilde{\lambda}})u_{\tilde{\lambda}}=0\;\;\hbox{and}\;\;J''(u_{\tilde{\lambda}})(u_{\tilde{\lambda}},u_{\tilde{\lambda}})\geq 0.$$
It follows from the last assertion Lemma \ref{N0lambda*} that $u_{\tilde{\lambda}}\in\mathcal{N}_{\tilde{\lambda}}^+$ holds for each $\tilde{\lambda}\in(0,\lambda^*)$.
Furthermore, for $\tilde{\lambda}=\lambda^*$, we observe that $\mathcal{N}_{\tilde{\lambda}}^0\neq \emptyset$. In light of Proposition \ref{sem sol} we have $u_{\tilde{\lambda}}\notin \mathcal{N}_{\tilde{\lambda}}^0$. Hence, we infer that $u_{\tilde{\lambda}}\in \mathcal{N}_{\tilde{\lambda}}^+.$

Now, we shall consider the sequence $(w_{\lambda_k})$. Due to Proposition \ref{Stein-Weiss} and Lemma \ref{imerso cont e comp}, we obtain
\begin{equation}\label{B(w_k)}
B(w_{\lambda_k})\leq C\Vert b\Vert_{\infty}\Vert w_{\lambda_k}\Vert^{2p}.
\end{equation}
In light of \eqref{t-lambda1} and \eqref{B(w_k)} we deduce that
$$C_1\Vert w_{\lambda_k}\Vert^{\frac{2-2p}{2p-2}}\leq t_\lambda(w_{\lambda_k})\leq t_{\lambda_k}^{n,-}(w_{\lambda_k})=1,\; u\in X\backslash\{0\}.$$
Now, taking the limit in the last estimate, we conclude that
$0<C_1\leq \Vert w_{\lambda}\Vert.$
Hence, we obtain that $w_{\lambda}\neq 0$. On the other hand, we know that $\{w_{\lambda_k}\}\in \mathcal{N}_{\lambda_k}^-$. Notice also that
$$J'(w_{\lambda_k})w_{\lambda_k}=0\;\;\hbox{and}\;\; J''(w_{\lambda_k})(w_{\lambda_k},w_{\lambda_k})<0.$$
As a consequence, we infer that
$$J'(w_{\tilde{\lambda}})w_{\tilde{\lambda}}=0\;\;\hbox{and}\;\; J''(w_{\tilde{\lambda}})(w_{\tilde{\lambda}},w_{\tilde{\lambda}})\leq 0.$$
Similarly, we conclude that $w_{\tilde{\lambda}}\in \mathcal{N}_{\tilde{\lambda}}^+$.

It remains to ensure that $\lambda \mapsto C_{\mathcal{N}_{\lambda}^-}$  is decreasing. Firstly, we observe that, for each
$\lambda_1 < \lambda_2$, we obtain the following estimates:
\begin{eqnarray}
    J_{\lambda_2}(tu) < J_{\lambda_1}(tu)\qquad\mbox{and}\qquad J_{\lambda_2}'(tu)(tu) < J_{\lambda_1}'(tu)(tu), \, \, u\in X\setminus\{0\},~ t > 0.\nonumber
\end{eqnarray}
Moreover, we know that
$C_{\mathcal{N}_{\lambda_1}^-} =  J_{\lambda_1}(t_{\lambda_1}^-(u_{\lambda_1})u_{\lambda_1})\;\;\;\;\mbox{and}\;\;\;\;C_{\mathcal{N}_{\lambda_2}^-} =  J_{\lambda_2}(t_{\lambda_2}^-(u_{\lambda_2})u_{\lambda_2}).$
Recall also that
\begin{eqnarray*}
\lambda_1=R_n(t_{\lambda_1}^\pm(u_{\lambda_1}) u_{\lambda_1})<R_n(t_{\lambda_2}^\pm(u_{\lambda_1}) u_{\lambda_1})=\lambda_2,
\end{eqnarray*}
Under these conditions, we are able to prove that
\begin{equation}\label{reltlamb}
    t_{\lambda_1}^+(u_{\lambda_1}) < t_{\lambda_2}^+(u_{\lambda_1}) < t_{\lambda_2}^-(u_{\lambda_1}) < t_{\lambda_1}^-(u_{\lambda_1}).
\end{equation}
Thus, we infer that
$
    C_{\mathcal{N}_{\lambda_2}^-} = J_{\lambda_2}(t_{\lambda_2}^-(u_{\lambda_2})u_{\lambda_2}) \le  J_{\lambda_2}(t_{\lambda_2}^-(u_{\lambda_1})u_{\lambda_1}) <  J_{\lambda_1}(t_{\lambda_2}^-(u_{\lambda_1})u_{\lambda_1})
        <  J_{\lambda_1}(t_{\lambda_1}^-(u_{\lambda_1})u_{\lambda_1}) = C_{\mathcal{N}_{\lambda_1}^-}.
$
Here, was used in the third inequality just above the fact that the fibering maps $t \mapsto J_{\lambda_1}(tu_{\lambda_1})$ is increasing on the interval $[t_{\lambda_1}^+(u_{\lambda_1}), t_{\lambda_1}^-(u_{\lambda_1})]$ and the relation given in \eqref{reltlamb}. Similarly, we prove that $\lambda \mapsto C_{\mathcal{N}_{\lambda}^+}$ is decreasing. Indeed, we mention that $t \mapsto J_{\lambda_2}(tu_{\lambda_1})$ is decreasing for $t \in (0, t_{\lambda_2}^+(u_{\lambda_1}))$. Now, by using \eqref{reltlamb}, we obtain that
\begin{eqnarray*}
    C_{\mathcal{N}_{\lambda_1}^+} & = & J_{\lambda_1}(t_{\lambda_1}^+(u_{\lambda_1})u_{\lambda_1}) > J_{\lambda_2}(t_{\lambda_1}^+(u_{\lambda_1})u_{\lambda_1}) > J_{\lambda_2}(t_{\lambda_2}^+(u_{\lambda_1})u_{\lambda_1}) \geq  J_{\lambda_2}(t_{\lambda_2}^+(u_{\lambda_2})(u_{\lambda_2}) = C_{\mathcal{N}_{\lambda_2}^+}.
\end{eqnarray*}
Thus, the functions $\lambda\mapsto C_{\mathcal{N}_{\lambda}^{\pm}}$ are decreasing.
At this stage, we shall prove that
$$\displaystyle\lim_{\lambda_k\to\tilde{\lambda}^-}C_{\mathcal{N}_{\lambda_k}^{\pm}}=C_{\mathcal{N}_{\tilde{\lambda}}^{\pm}}.$$
Without loss of generality, we assume  $C_{\mathcal{N}_{\tilde{\lambda}}^+}=J_{\tilde\lambda}(u_{\tilde{\lambda}})$ and $C_{\mathcal{N}_{\tilde{\lambda}_k}^+}=J_{\tilde\lambda}(u_{\tilde{\lambda}_k})$.  Notice also that $\lambda\mapsto C_{\mathcal{N}_{\tilde{\lambda}}^+}$ is decreasing and $\lambda_k\leq \tilde\lambda$. Hence, we deduce that
$$C_{\mathcal{N}_{\tilde{\lambda}}^+}\leq \liminf_{\lambda_k\to \tilde\lambda} C_{\mathcal{N}_{{\lambda_k}}^+}.$$
On the other hand, by using \eqref{4.1}, we see that
$$\displaystyle\limsup_{\lambda_k\to \tilde{\lambda}} C_{\mathcal{N}_{\lambda_k}^+}\leq C_{\mathcal{N}_{\tilde{\lambda}}^+}.$$
Thus, we obtain
$\lim_{\lambda_k\to \tilde{\lambda}} C_{\mathcal{N}_{\lambda_k}^+}=C_{\mathcal{N}_{\tilde{\lambda}}^+}.$
In particular, we infer that
$\lim_{\lambda_k \to \lambda^*}C_{\mathcal{N}_{\lambda_k}^+} = C_{\mathcal{N}_{\lambda^*}^+}.$
By a similar argument, we prove that $\lim\limits_{\lambda_k \to \lambda^*}C_{\mathcal{N}_{\lambda_k}^-} = C_{\mathcal{N}_{\lambda^*}^-}$. This ends the proof.
\end{proof}
\end{prop}
\begin{theorem}\label{salva}
 Suppose ($H_0$), ($H_1'$), ($H_2'$), ($V_1$) and ($V_2$). Then the Problem \eqref{P1} admits at least two solutions denoted by $w_{\lambda^*}\in \mathcal{N}^-_{\lambda^*}$ and $u_{\lambda^*}\in\mathcal{N}^+_{\lambda^*}.$
\begin{proof}
Let us consider $(\lambda_k)\subset (0,\lambda^*)$ be a sequence such that $\lambda_k\to\lambda^*$. Assume also that $(w_{\lambda_k})\subset \mathcal{N}^-_{\lambda_k}$ is a sequence of positive weak solutions for the Problem \eqref{P1} with $\lambda = \lambda_k$. Recall that $(w_{\lambda_k})$ is a bounded sequence in $X$. Hence, $w_{\lambda_k}\rightharpoonup w_{\lambda^*}$ in $X$. Thus, there exists $h_m\in L^m(\mathbb R^N)$ such that $0\leq w_{\lambda_k}  \leq h_m$, $w_{\lambda_k} \rightarrow w_{\lambda^*}~\hbox{in}~ L^m(\mathbb R^N)$ and $w_{\lambda_k}(x) \rightarrow w_{\lambda^*}(x) \ \hbox{a.e in}\;\mathbb R^N$, $2\leq m< 2^*$. As a consequence,  we have that $w_{\lambda^*}\geq 0$.

In virtue of Proposition \ref{C_n decrescente} item $a),$ we obtain that $w_{\lambda_k}\to w_{\lambda^*}$ in $X$. This leads us to prove that
 \begin{equation}\label{E14}
     \phi'_{\lambda^*,w_{\lambda^*}}(1)=\displaystyle\lim_{k\to \infty}\phi'_{\lambda_k,w_{\lambda_k}}(1)=0\qquad\mbox{and}\qquad \phi''_{\lambda^*,w_{\lambda^*}}(1)=\displaystyle\lim_{k\to \infty}\phi''_{\lambda_k,w_{\lambda_k}}(1)\leq 0.
 \end{equation}
Therefore, by using \eqref{E14}, we infer that $w_{\lambda^*}\in \mathcal{N}_{\lambda^*}^-\cup \mathcal{N}_{\lambda^*}^0$.

Now, we shall prove that $w_{\lambda^*}$ is a weak solution for the Problem \eqref{P1}. Recall that $w_{\lambda_k}$ is a weak solution for the Problem ($Q_{\lambda_k}$). Thus, we obtain that
 \begin{equation*}\label{sol. fraca lambida*1}
     \left<w_{\lambda_k},\varphi\right>-D(w_{\lambda_k},\varphi)=\lambda_kH(w_{\lambda_k},\varphi),\;\varphi\in X.
 \end{equation*}
The last assertion implies that
 \begin{equation*}
 \displaystyle\liminf_{k\to\infty}\left[\left<w_{\lambda_k},\varphi\right>-D(w_{\lambda_k},\varphi)\right]=\displaystyle\liminf_{k\to\infty}\lambda_kH(w_{\lambda_k},\varphi).
 \end{equation*}
Hence, applying Fatou's Lemma, we deduce that
\begin{equation}\label{w-posi}
 \left<w_{\lambda^*},\varphi\right>-D(w_{\lambda^*},\varphi)\geq \lambda^*\displaystyle\int_{\mathbb R^N}a(x)G(x)\varphi dx,
\end{equation}
holds for every $\varphi\in X_+$ where
 \begin{eqnarray*}
 G(x):=\left\{\begin{array}{ll}
      w_{\lambda^*}^{q-1},\;\;\hbox{se}\;\;w_{\lambda^*}(x)\neq 0,\\
    \infty,\;\;\hbox{se}\;\;w_{\lambda^*}(x)=0.
      \end{array}
\right.
\end{eqnarray*}
It follows from \eqref{w-posi} that
$$0<\displaystyle\int_{\mathbb R^N}a(x)G(x)\varphi dx<\infty.$$
Under these conditions, we deduce that $w_{\lambda^*}(x)> 0$ a.e. in $\mathbb R^N$ and $aw_{\lambda^*}^{q-1}\varphi\in L^1(\mathbb R^N)$ holds for all $\varphi\in X_+$.
As a product, we have that
\begin{equation}\label{>0}
\displaystyle\int\limits_{\mathbb R^N}(\nabla w_{\lambda^*}\nabla \varphi+V(x)w_{\lambda^*}\varphi) dx- \lambda H(w_{\lambda^*},\varphi)- D(w_{\lambda^*},\varphi)\geq 0,\;\;\varphi\in X_+.
\end{equation}
Now, we shall prove that
$$\displaystyle\int\limits_{\mathbb R^N}(\nabla w_{\lambda^*}\nabla \varphi+V(x)w_{\lambda^*}\varphi) dx- \lambda H(w_{\lambda^*},\varphi)- D(w_{\lambda^*},\varphi)= 0,~\varphi\in X.
$$
In order to do that, given any $\varphi \in X$ and $\varepsilon>0$ small enough, we define $\Psi=(w_{\lambda^*}+\varepsilon\varphi)^+\in X_+$. Recall that $\Psi=(w_{\lambda^*}+\varepsilon\varphi)^+=(w_{\lambda^*}+\varepsilon\varphi)-(w_{\lambda^*}+\varepsilon\varphi)^-.$
Under these conditions, by using $\Psi$ as a test function in \eqref{>0}, we deduce that
\begin{eqnarray*}
     0\nonumber &\leq & \left< w_{\lambda^*},w_{\lambda^*}+\varepsilon\varphi\right>-\lambda^*\displaystyle\int_{\mathbb R^N}a(x) w_{\lambda^*}^{q-1}(w_{\lambda^*}+\varepsilon\varphi)dx\\
    \nonumber&&-\displaystyle\int_{\mathbb R^N}\displaystyle\int_{\mathbb R^N}\displaystyle\frac{b(y) w_{\lambda^*}^pb(x) w_{\lambda^*}^{p-1}(w_{\lambda^*}+\varepsilon\varphi)}{\vert x\vert^{\alpha}\vert x-y\vert^{\mu}\vert y\vert^{\alpha}}dxdy-\left<w_{\lambda^*},(w_{\lambda^*}+\varepsilon\varphi)^-\right>\\
\nonumber&&+\lambda^*\displaystyle\int_{(w_{\lambda^*}+\varepsilon\varphi\leq 0)}a(x) w_{\lambda^*}^{q-1}(w+\varepsilon\varphi)dx+\displaystyle\int_{(w_{\lambda^*}+\varepsilon\varphi\leq 0)}\displaystyle\int_{\mathbb R^N}\displaystyle\frac{b(y) w_{\lambda^*}^pb(x) w_{\lambda^*}^{p-1}(w_{\lambda^*}+\varepsilon\varphi)}{\vert x\vert^{\alpha}\vert x-y\vert^{\mu}\vert y\vert^{\alpha}}dxdy\\
\nonumber &\leq & \Vert w_{\lambda^*}\Vert^2-\lambda^*\displaystyle\int_{\mathbb R^N}a(x) w_{\lambda^*}^q dx-\displaystyle\int_{\mathbb R^N}\displaystyle\int_{\mathbb R^N}\displaystyle\frac{b(y) w_{\lambda^*}^pb(x) w_{\lambda^*}^p}{\vert x\vert^{\alpha}\vert x-y\vert^{\mu}\vert y\vert^{\alpha}}dxdy -\varepsilon\int_{\left(w_{\lambda^*}+\varepsilon\varphi\right)\leq 0}\nabla w_{\lambda^*}\nabla\varphi+V(x)w_{\lambda^*}\varphi dx\\
\nonumber&&+\varepsilon\left[\left<w_{\lambda^*},\varphi\right>-\lambda^*\displaystyle\int_{\mathbb R^N}a(x) w_{\lambda^*}^{q-1}\varphi dx-\left.\displaystyle\int_{\mathbb R^N}\displaystyle\int_{\mathbb R^N}\displaystyle\frac{b(y) w_{\lambda^*}^p b(x) w_{\lambda^*}^{p-1}\varphi}{\vert x\vert^{\alpha}\vert x-y\vert^{\mu}\vert y\vert^{\alpha}}dxdy\right.\right].\\
\end{eqnarray*}
On the other hand, we know that $w_{\lambda^*}\in \mathcal{N}_{\lambda^*}$. Hence, we verify that
\begin{eqnarray*}
0 &\leq &\left<w_{\lambda^*}, \varphi \right>-\lambda^*H(w_{\lambda^*},\varphi)-D(w_{\lambda^*},\varphi)-\displaystyle\int_{\left(w_{\lambda^*}+\varepsilon\varphi\right)\leq 0}\nabla w_{\lambda^*}\nabla\varphi+V(x)w_{\lambda^*}\varphi\; dx,\;\;\varphi\in X.
\end{eqnarray*}
Similarly, using the same ideas just discussed in Lemma \ref{3 passos}, we deduce also that
\begin{equation*}\label{m5}
    \displaystyle\lim_{\varepsilon\to 0}\displaystyle\int_{\R^N}[\nabla w_{\lambda^*}\nabla\varphi+V(x)w_{\lambda^*}\varphi ]\chi_{\left(w_{\lambda^*}+\varepsilon\varphi\right)\leq 0}\; dx=0.
\end{equation*}
Thus, we obtain that
\begin{equation*}\label{m6}
  0\leq \left < w_{\lambda^*},\varphi\right>-\lambda^*H(w_{\lambda^*},\varphi)-D(w_{\lambda^*},\varphi),\,\, \varphi\in X.
\end{equation*}
As a consequence, we infer that
\begin{eqnarray*}\label{m8}
  \left<w_{\lambda^*}, \varphi \right>-\lambda^*H(w_{\lambda^*},\varphi)-D(w_{\lambda^*},\varphi)=0.
\end{eqnarray*}
Therefore, $w_{\lambda^*}$ is a weak solution for the Problem \eqref{P1} and $ w_{\lambda^*}\in \mathcal{N}_{\lambda^*}^-\cup \mathcal{N}_{\lambda^*}^0$.
In light of Proposition \ref{sem sol} we obtain that $ w_{\lambda^*}\in \mathcal{N}_{\lambda^*}^-$. A similar argument shows that $u_{\lambda^*}\in \mathcal{N}_{\lambda^*}^+$ is a weak solution for the Problem \eqref{P1}. This completes the proof.
  \end{proof}
\end{theorem}

\subsection{The proof of Theorem \ref{teorema 2} completed} The proof of Theorem \ref{teorema 2} follows by using Proposition \ref{sem sol} and Theorem \ref{salva}.

\section{\textbf{Declarations}}

\textbf{Ethical Approval }

It is not applicable.

\textbf{Competing interests}

There are no competing interests.

\textbf{Authors' contributions}

All authors wrote and reviewed this paper.

\textbf{Funding}

The second author was partially supported by CNPq with grants 309026/2020-2. CNPq partially supported the fourth author with grant 316643/2021-1. Minbo Yang was partially supported by NSFC (12471114) and ZJNSF(LZ22A010001).

\textbf{Availability of data and materials}

All the data can be accessed from this article.

\end{document}